\documentclass[11pt]{amsart}
\usepackage{a4}
\usepackage{amssymb}
\usepackage{amsthm}
\usepackage{amsmath}
\usepackage{amsopn}
\usepackage{array,enumerate}
\usepackage{verbatim}

\DeclareMathOperator{\supp}{supp}

\DeclareMathOperator{\diam}{diam}
\DeclareMathOperator{\Id}{Id}
\DeclareMathOperator{\Tr}{Tr}
\newcommand{\per}{\mathrm{per}}
     
     \newcommand{\EE}{\mathbb{E}}
     
     \newcommand{\NN}{\mathbb{N}}
     \newcommand{\PP}{\mathbb{P}}
     
     \newcommand{\RR}{\mathbb{R}}
     \newcommand{\ZZ}{\mathbb{Z}}

\newcommand{\cF}{\mathcal{F}}
\newcommand{\cP}{\mathcal{P}}

\newcommand{\be}{\begin{equation}}
\newcommand{\ee}{\end{equation}}
\newcommand{\la}{\langle}
\newcommand{\ra}{\rangle}

\newcommand{\KKK}{S}
\newcommand{\SSS}{K}

\theoremstyle{plain}
\newtheorem{teo}{Teorem}
\newtheorem{theorem}[teo]{Theorem}
\newtheorem{proposition}[teo]{Proposition}
\newtheorem{lemma}[teo]{Lemma}

\theoremstyle{definition}
\newtheorem{definition}[teo]{Definition}

\theoremstyle{remark}
\newtheorem{remark}[teo]{Remark}

\newcommand{\change}{\leavevmode{\marginpar{\tiny%
$\hbox to 0mm{\hspace*{-0.5mm}$\leftarrow$\hss}%
\vcenter{\vrule depth 0.1mm height 0.1mm width \the\marginparwidth}%
\hbox to
0mm{\hss$\rightarrow$\hspace*{-0.5mm}}$\\\relax\raggedright }}}

\newcommand{\Hm}[1]{\leavevmode{\marginpar{\tiny%
$\hbox to 0mm{\hspace*{-0.5mm}$\leftarrow$\hss}%
\vcenter{\vrule depth 0.1mm height 0.1mm width \the\marginparwidth}%
\hbox to
0mm{\hss$\rightarrow$\hspace*{-0.5mm}}$\\\relax\raggedright #1}}}

\begin{document}

\title[Equality of Lifshitz and van Hove exponents]{Equality of Lifshitz and van Hove exponents on amenable Cayley graphs}
\author[T.~Antunovi\'c]{Ton\'ci Antunovi\'c}
\address{Department of Mathematics, UC Berkeley, Berkeley, CA 94720, USA}
\author[I.~Veseli\'c]{Ivan Veseli\'c}
\address{Emmy-Noether-Programme of the Deutsche Forschungsgemeinschaft\vspace*{-0.3cm} }
\address{\& Fakult\"at f\"ur Mathematik,\, 09107\, TU\, Chemnitz, Germany   }
\urladdr{www.tu-chemnitz.de/mathematik/schroedinger/members.php}


\begin{abstract}
We study the low energy asymptotics of periodic and random Laplace operators
on Cayley graphs of amenable, finitely generated groups.
For the periodic operator the asymptotics is characterised by the van Hove exponent or zeroth Novikov-Shubin invariant.
The random model we consider is given in terms of an adjacency Laplacian
on site or edge percolation subgraphs of the Cayley graph.
The asymptotic behaviour of the spectral distribution is exponential, characterised by the
Lifshitz exponent. We show that for the adjacency Laplacian the two invariants/exponents coincide.
The result holds also for more general symmetric transition operators.
For combinatorial Laplacians one has a 
different universal behaviour of the low energy asymptotics of the spectral distribution function,
which can be actually established on quasi-transitive graphs without an amenability assumption.
The latter result holds also for long range bond percolation models.
\end{abstract}

 \thanks{\jobname.tex; \today. To appear in \emph{Journal de Mathematiques Pures et Appliquees}.}
\maketitle

\section{Introduction} 
Operators on Euclidean space which are invariant under a group action
have a well defined \emph{integrated density of states} (IDS),
also known as the \emph{spectral distribution function}.
Prominent examples are Laplace and Schr\"odinger operators.
Their IDS exhibits a \emph{van Hove singularity} at the bottom of the spectrum.
This means that it vanishes polynomially as the energy parameter approaches the lowest spectral edge,
the exponent being equal to the space dimension divided by two.
The factor one half is due to the fact that the considered operators are
elliptic of second order.

The IDS can be defined also for operators having a more general
type of equivariance property, namely for ergodic operators.
Two prominent classes of such operators are random and almost periodic ones.
Among the pioneering works which have studied the IDS of such models are \cite{Pastur-71},
respectively \cite{Shubin-78}.

Several well-studied types of random operators on $L^2(\RR^d)$ and
$\ell^2(\ZZ^d)$  exhibit a Lifshitz tail at the bottom of the
spectrum, meaning that the IDS vanishes exponentially fast.
In particular, the spectral density is very sparse in this region
and spectral values are created only by extremely rare configurations of the randomness.
Hence such spectral edges are called \emph{fluctuation boundaries}.
%
% This implies that the IDS is continuous at the bottom of the spectrum,
% which means that the zeroth $L^2$-Betti number of the ergodic operator vanishes.
%
In Euclidean space the Lifshitz exponent is quite universal. In particular, for Laplacians with a variety of
random i.i.d. non-negative perturbations it equals $d/2$,
cf.~for instance the survey \cite{KirschM-07} and the references therein.

Historically, physicists have introduced the IDS as a limit of
spectral distribution functions of finite volume operators.
For this approximation to converge, the underlying space or group needs to have
some amenability property.  However, for the purposes of the present
paper the approximation property is not relevant and we may rather consider the
IDS as given by a \emph{Shubin-Pastur trace formula} \eqref{e-IDS}.

In the present paper we want to analyse whether the Lifshitz
exponent equals the van Hove exponent for operators on more
general geometries as well. Of course, for this to hold a proper relation between the considered periodic and
random operator is necessary, in the sense that the random operator results
from its periodic counterpart  by addition of
stochastically independent, positive perturbations. The periodic
objects we study are Laplace operators on Cayley graphs. We
consider two different types of random perturbations thereof: the
adjacency and the combinatorial Laplacians on random subgraphs
generated by a subcritical percolation process. While the first
type of operators indeed shows a coincidence of van Hove and
Lifshitz exponents, the second ones exhibit a different type of
universal behaviour, the reason being, that the random
perturbation is not positive in this case.

Our motivation to study this question is threefold: firstly, to extend the
results of \cite{KloppN-03} and \cite{KirschM-06} concerning lattice bond percolation models;
secondly, to study the relation between van Hove and Lifshitz exponents, as
done at internal spectral edges of random perturbations of periodic Schr\"odinger operators
e.g.~in \cite{Klopp-99,KloppW-02}, and finally to clarify some of the links between
geometric $L^2$-invariants and the IDS, see
e.g.~\cite{Lueck-02,DodziukLPSV-06}. Note in particular that the
van Hove exponent equals  the Novikov-Shubin invariant of
order zero, cf.~\cite{NovikovSh-86,GromovS-91}. Our strategy of proof is coined after the one in
 \cite{KirschM-06}. The description of the asymptotic
behaviour of the IDS at spectral boundaries of random operators
plays a key role in the proof of Anderson localisation, see
e.g.~\cite{GerminetK-03}. For more background on the IDS of
percolation Hamiltonians on Cayley graphs see the discussion in
\cite{AntunovicV-b}.
%Let us mention that low energy asymptotics of random Schr\"odinger operators
%on hyperbolic space have been studied  in \cite{Sznitman-89}.

In the next section we state our theorems. Thereafter, in Section
\ref{s-abstract-bounds} we present abstract upper and lower bounds
on the IDS. Section~\ref{s-eigenvalue-bounds} is devoted to
eigenvalue inequalities.  Section~\ref{s-polynomial-growth}
contains the proofs of the theorems in the case of adjacency
Laplacians on groups with polynomial growth and combinatorial
Laplacians on general quasi-transitive graphs. In Section~\ref{s-lamplighter}
we prove the statements concerning Lamplighter groups.
The last section is devoted to the extension of our results to some 
related models: we derive there the low energy  
spectral asymptotics of percolation Hamiltonians associated to
general symmetric transition operators
on discrete, finitely generated, amenable groups. Furthermore 
we study combinatorial Laplacians on long range edge percolation graphs,
and on an abstract ensemble of percolation graphs satisfying certain conditions.

\section{Definitions and results}
\label{s-results}
We describe the type of graphs, the percolation process and the operators we will be considering.

Let $\Gamma$ be a discrete, finitely generated group, $S$ a
finite, symmetric set of generators not containing the unit
element $\iota$ of $\Gamma$ and $G=(V,E)$ the associated Cayley
graph. It is $k$-regular with $k=|S|$. The ball around $\iota$ of
radius $n$ is denoted by $B(n)$ and its volume by $V(n)$. From
\cite{Bass-72,Gromov-81,vandenDriesW-84} it is known that either
there are $d \in \NN, a,b >0  $ such that $a\, n^{d} \leq V(n)
\leq b\, n^{d}$, in which case $\Gamma$ is called to be of
\emph{polynomial growth of order $d$}; or for every $d\in\NN$ and
every $b\in\RR$ there exist only finitely many integers $n$ such
that $V(n) \leq  b\, n^{d}$, in which case $\Gamma$ is called to
be of \emph{superpolynomial growth}. The growth type depends only
on the group and not on the choice of the set of generators used
to define the Cayley graph. Cayley graphs are a particular case of
\emph{quasi-transitive} graphs, i.e.~graphs whose vertex set decomposes under
the action of the automorphism group into finitely many orbits.
Most of our results are valid only for Cayley graphs. An exception
are the Theorems which concern the combinatorial Laplacian
(e.g.~Theorem \ref{thm: main Neumann}) which apply to general
quasi-transitive graphs with finite vertex degree.

Next we introduce site percolation on infinite, connected, quasi-transitive graphs.
For $p \in [0,1]$, let $\omega_x, x \in V$ be an i.i.d.~sequence of Bernoulli random variables each taking the
value $1$  with probability $p$ and the value $0$  with probability $1-p$.
The set of possible configurations $\omega =(\omega_x)_{x\in V}$ is denoted by $\Omega$
and the corresponding product probability measure with $\PP$.
We call $V(\omega):=\{x \in V\mid \omega_x =1\}$ the set of \emph{open sites}. The induced subgraph of $G$
with vertex set $V(\omega)$ is denoted by $G_\omega$ and called the \emph{percolation subgraph}
in the configuration $\omega$. The connected components of $G_\omega$ are called \emph{clusters}.
For a fixed vertex $o \in V$ we denote by $C_o(\omega)$ the connected component which contains it.
The bond percolation process is defined analogously. In this case the percolation subgraph
$G_\omega$ is the graph whose edge set $E(\omega)$ is the set of all $e\in E$ with $\omega_e =1$ and whose
vertex set $V(\omega)$ consist of all vertices in $V$ which are incident to an element of $E(\omega)$.
For both site and bond percolation there exists a \emph{critical parameter} $0<p_{c}\leq 1$
such that for $p<p_{c}$ there is no infinite cluster almost surely and for
$p>p_{c}$ there is an infinite cluster almost surely. The first case is called the \emph{subcritical phase}
and the second \emph{supercritical phase}. The theorems of this paper concern only the subcritical percolation phase.
We will denote the expectation with respect to  $\PP$ by $\EE \{  \dots\}$.

In the following we assume throughout that $G$ is an infinite, countable 
quasi-transitive graph with bounded vertex degree and that 
there exist a group of automorphisms acting freely and cofinitely on $G$. In particular
it may be a Cayley graph. Let $G'=(V',E')$ be an arbitrary
subgraph of  $G$, possibly $G$ itself. Note that even if $G$ is
regular, $G'$ need not be. We denote the degree of the vertex
$x\in V'$ in $G'$ by $\deg_{G'}(x)$. If two vertices $x,y\in V'$
are adjacent in the subgraph $G'$ we write  $y \sim_{G'} x$.

For $G$ and $G'$ as above we define the following operators on $\ell^2(G'):=\ell^{2}(V')$.
\begin{definition}\label{def: operators on graphs}
\begin{enumerate}[(a)]
\item The identity operator on $\ell^2(V')$ is denoted by $\Id$.
\item The \emph{degree operator} acts on $\varphi \in\ell^{2}(V')$ according to
 \[
 [D(G')\varphi](x) := \deg_{G'}(x)\varphi(x).
 \]
\item The \emph{adjacency operator} is defined as
 \[
 \displaystyle [A(G')\varphi](x) := \sum_{y \in V', y \sim_{G'} x}\varphi(y).
 \]
\item The \emph{combinatorial Laplacian} is defined as
 \[
 H^{\scriptscriptstyle N}(G') :=  D(G') - A(G').
 \]
\end{enumerate}
\end{definition}

If $G$ is a $k$-regular graph we define additionally:
\begin{definition}\label{def: operators on regular graphs}
\begin{enumerate}[(a)] \setcounter{enumi}{4}
\item The \emph{adjacency Laplacian}  on $G'$ is defined as
 \[
 H^{\scriptscriptstyle A}(G') := k \Id - A(G').
 \]
\item
The \emph{boundary potential} is the multiplication operator 
\[
W^{b.c.}(G')= k \Id -D(G')
\]
\item The \emph{Dirichlet Laplacian} is defined as
 \[
 H^{\scriptscriptstyle D}(G') := H^{\scriptscriptstyle A}(G')  +W^{b.c.}(G')= 2k \Id - D(G') - A(G').
 \]
\end{enumerate}
Note that $H^{\scriptscriptstyle N}(G') = H^{\scriptscriptstyle A}(G') - W^{b.c.}(G')$. Of course, it is
possible to define the operators (e) -- (g) also for non regular
graphs, but then there is no canonical choice for the value $k$
which would give them a geometric meaning.
\end{definition}
It follows that the quadratic forms of the combinatorial, Dirichlet, and adjacency Laplacian are given by
\begin{align}
\label{e-quadratic-forms}
\la  H^{\scriptscriptstyle N}(G')\phi, \phi\ra &= \sum_{(x,y)\in E'} |\phi(x)-\phi(y)|^2
\nonumber\\
\la  H^{\scriptscriptstyle D}(G')\phi, \phi\ra &= 2 \sum_{x \in V'}( k- \deg_{G'}(x))\, |\phi(x) |^2 + \sum_{(x,y)\in E'} |\phi(x)-\phi(y)|^2
\\ \nonumber
\la  H^{\scriptscriptstyle A}(G')\phi, \phi\ra &=  \sum_{x \in V'}( k- \deg_{G'}(x))\, |\phi(x) |^2 + \sum_{(x,y)\in E'} |\phi(x)-\phi(y)|^2
\end{align}
and satisfy $H^{\scriptscriptstyle N}(G') \le H^{\scriptscriptstyle A}(G')\le H^{\scriptscriptstyle D}(G')$ in the sense of quadratic forms.

\begin{remark}[Terminology]
If $G'=G$ and $G$ is regular then the operators $ H^{\scriptscriptstyle A},H^{\scriptscriptstyle N},H^{\scriptscriptstyle D}$
coincide and we denote them simply by $H$.  If $G$ is the Cayley graph of an amenable group
the spectral bottom of $H$ equals zero. Usually in the graph theory literature the adjacency
matrix and the combinatorial Laplacian are the objects of study. For the first operator one is (among others)
interested in the properties related to the upper edge of the spectrum,
whereas for the second operator one considers the low-lying spectrum. In order to be able to treat
both operators in parallel it is convenient to consider $H^{\scriptscriptstyle A}$ rather than $A$. Of course,
spectral properties of  $H^{\scriptscriptstyle A}$ directly translate to those of $A$.

Motivated by the Dirichlet-Neumann bracketing for Laplacians in
the continuum, in \cite{Simon-85b} the terminology of Neumann
$H^{\scriptscriptstyle N}$ and Dirichlet $H^{\scriptscriptstyle D}$ Laplacians was introduced. This is
the reason why we use the superscript $N$ for the combinatorial
Laplacian. While in the continuum the boundary conditions are
necessary to define a selfadjoint operator, in the discrete
setting they correspond to a boundary potential $W^{b.c.}$, which
is either added or subtracted to/from the Laplacian without
boundary term, i.e.~the adjacency Laplacian $H^{\scriptscriptstyle A}$. Note however,
that the term Neumann Laplacian is sometimes, e.g.~in
\cite{Chung-97}, used for a different operator. Likewise, the
operator $H^{\scriptscriptstyle A}$ is often called Dirichlet Laplacian, e.g.~in
\cite{ChungGY-00}, while in \cite{KirschM-06} it is called
Pseudo-Dirichlet Laplacian.
\end{remark}

Given a (site or bond) percolation subgraph $G_\omega\subset G$
we use the following abbreviations for operators on $\ell^2(V(\omega))$:
$\deg_\omega(x)= \deg_{G_\omega} (x), A_\omega = A(G_\omega),
H^{\scriptscriptstyle A}_\omega=H^{\scriptscriptstyle A}(G_\omega),H^{\scriptscriptstyle N}_\omega=H^{\scriptscriptstyle N}(G_\omega),H^{\scriptscriptstyle D}_\omega=H^{\scriptscriptstyle D}(G_\omega),
W^{b.c.}_\omega=W^{b.c.}(G_\omega)$. Any one of the  operators $H^{\scriptscriptstyle \#}_\omega, \#\in\{A,N,D\}$ will be
called a percolation Laplacian. If $G$ is a Cayley graph we consider all three types
$H^{\scriptscriptstyle A},H^{\scriptscriptstyle N},H^{\scriptscriptstyle D}$, while in the case of a quasi-transitive graph we will derive results only
for the combinatorial Laplacian $H^{\scriptscriptstyle N}$.

Next we define the IDS. Let $G$ be a quasi-transitive graph equipped with a subgroup $\Gamma$ of its
automorphism group which acts freely and cofinitely  on $G$. Denote by $\cF$ an arbitrary, but fixed
$\Gamma$-fundamental domain, i.e.~a subset of $G$, which contains exactly one element of each $\Gamma$-orbit.
The IDS of the random operator $(H^{\scriptscriptstyle \#}_\omega)_\omega$ may be defined by the following trace formula:
%
% (if $\# \in \{A,D\}$, we assume that $G$ is a Cayley graph)
%
\begin{equation}\label{e-IDS}
 N^{\scriptscriptstyle \#}(E):= \frac{1}{|\cF|}\EE\big\{\Tr [\chi_\cF \, \chi_{]-\infty, E]}(H^{\scriptscriptstyle \#}_\omega)]\big\}.
\end{equation}
Here $\chi_\cF $ is understood to be a multiplication operator.
If $\Gamma$ acts transitively on $G$ the expression \eqref{e-IDS}
simplifies to $\EE\{ \la\delta_x, \chi_{]-\infty, E]}(H^{\scriptscriptstyle \#}_\omega)\delta_x\ra\}$,
where $x$ denotes an arbitrary vertex in $G$ and $\delta_x$ its characteristic function.
If moreover $p=1$, i.e.~ we consider the IDS of the Laplacian $H$ on $G$ itself,
the formula simplifies further to
$N_\per(E)=\la\delta_x, \chi_{]-\infty, E]}(H)\delta_x\ra$.
We denote by $N_\per$ the IDS of the periodic operator $H$, while $N^{\scriptscriptstyle \#}$
is reserved for the IDS of the random operator $H^{\scriptscriptstyle \#}$.

\begin{remark} \label{r-basicproperties}
Several properties of the random family $(H^{\scriptscriptstyle \#}_\omega)_\omega$
of operators play a role in the definition of the IDS. These hold for any
of the boundary types $\#\in \{A,N,D\}$. Firstly, $(H^{\scriptscriptstyle \#}_\omega)_\omega$
is a measurable family of operators in the sense of \cite{LenzPV-07} (which extends the
notion introduced in \cite{KirschM-82c}). Secondly, each operator is bounded,  selfadjoint
and non-negative.

If the group $\Gamma$ is amenable, it is possible to approximate the IDS by its analogs associated to
operators restricted to finite graphs along a F\o lner (van Hove) sequence. This has been shown
for periodic operators in \cite{DodziukMY,MathaiY-02} and for site percolation Hamiltonians in \cite{Veselic-05a}.
For bond percolation Hamiltonians the same proof applies. For bond percolation on the lattice 
$\ZZ^d$ these results
were proven in \cite{KirschM-06}.

A finitely generated, discrete group is amenable if and only if it contains an increasing F\o lner sequence,
i.e.~an increasing sequence of finite subsets $I_n\subset \Gamma$ such that
$$\lim_{j \to \infty} \, \frac{|I_{j}\,\triangle \,F \cdot I_{j}|}{|I_{j}|} = 0, \text{ \ for any finite } F \subset \Gamma.$$
Any increasing F\o lner sequence induces a  monotone exhaustion
$\Lambda_n, n\in \NN$ consisting of finite subsets $\Lambda_n$ of
the vertex set of $G$, such that if we denote by $H^{\scriptscriptstyle \#,n}_\omega$
the restriction of  $H^{\scriptscriptstyle \#}_\omega$ to $\ell^2(\Lambda_n \cap
V(\omega))$ the convergence
\begin{equation}
\label{e-appoximationIDS}
\lim_{n\to\infty}
\frac{1}{|\Lambda_n|} \, \Tr [ \chi_{]-\infty, E]}(H^{\scriptscriptstyle \#,n}_\omega)] =N^{\scriptscriptstyle \#}(E)
\end{equation}
holds for almost all $\omega$ and all continuity points $E$ of $N^{\scriptscriptstyle \#}$. 
Actually, for this convergence one has to assume that the F\o lner sequence $(I_j)_j$ is tempered,
cf.~the ergodic theorem in \cite{Lindenstrauss-01}. This is no loss of generality, since every amenable group contains a tempered F\o lner sequence.
Note that
since $H^{\scriptscriptstyle \#,n}_\omega$ is a finite dimensional operator its spectrum consists
entirely of eigenvalues
%$\lambda_1 \le \dots \lambda_D, D = |\Lambda_n \cap V(\omega)|$
and hence
$\Tr [ \chi_{]-\infty, E]}(H^{\scriptscriptstyle \#,n}_\omega)]$ equals the number of eigenvalues of
$H^{\scriptscriptstyle \#,n}_\omega$ not exceeding $E$. Using the inequalities for the quadratic forms \eqref{e-quadratic-forms}
and Weyl's monotonicity principle it follows that
$\Tr [ \chi_{]-\infty, E]}(H^{\scriptscriptstyle N,n}_\omega)]\ge \Tr [ \chi_{]-\infty, E]}(H^{\scriptscriptstyle A,n}_\omega)]
\ge \Tr [ \chi_{]-\infty, E]}(H^{\scriptscriptstyle D,n}_\omega)]$. Passing to the limit $n \to \infty$
one obtains $N^{\scriptscriptstyle N} \ge N^{\scriptscriptstyle A} \ge N^{\scriptscriptstyle D}$.
Recently it turned out that the statement in \eqref{e-appoximationIDS}
can be strenghtened: in \cite{LenzV} it was shown that the convergence holds \emph{uniformly}
with respect to the energy parameter $E$.

There are other important properties of $(H^{\scriptscriptstyle \#}_\omega)_\omega$ which are appropriate to mention here
although they are not necessary for the formulation of our definitions or theorems.
The spectrum of $H^{\scriptscriptstyle \#}_\omega$ is almost surely $\omega$-independent, cf.~\cite{LenzPV-07,Veselic-05a}.
We denote it by $\Sigma^{\scriptscriptstyle \#}$ in the sequel. The same holds for the measure-theoretic components of the spectrum.
The topological support of the measure whose distribution function is $N^{\scriptscriptstyle \#}$ coincides with $\Sigma^{\scriptscriptstyle \#}$.
Using the same arguments as in \cite{KirschM-06} one can show that $\Sigma^{\scriptscriptstyle \#}\supset \sigma(H)$.
The IDS of percolation Hamiltonians has a rich set of discontinuities \cite{ChayesCFST-86},
and a characterisation of this set is given in \cite{Veselic-05b}. It is also possible to extend the percolation Hamiltonians
to the removed vertices $V\setminus V(\omega)$ by a constant. This is just a matter of convention and does not
alter the results essentially.  For a broader discussion of the above facts see \cite{AntunovicV-b}.
\end{remark}

The next statement characterises the asymptotic behaviour of the IDS of the periodic Laplacian $H$
at the spectral bottom and can be inferred from \cite{Lueck-02,Varopoulos-87}.
For groups of polynomial growth it exhibits a van Hove singularity, while
in the case of superpolynomial growth one encounters a  different type of asymptotics
which may be interpreted as corresponding to a van Hove exponent equal to infinity.

\begin{theorem}\label{thm: full graph}
Let $\Gamma$ be an infinite,  finitely generated,  amenable group, $H$ the Laplace operator on a Cayley graph of $\Gamma$
and $N_\per$ the associated IDS. If $\Gamma$ has polynomial growth of order $d$ then
\begin{align}
\label{e-vanHove}
    \lim_{E \searrow 0} \frac{\ln N_\per(E)}{\ln  E} &= \frac{d}{2}.
\intertext{and if $\Gamma$ has superpolynomial growth then}
    \lim_{E \searrow 0} \frac{\ln N_\per(E)}{\ln  E} &= \infty.
\end{align}
\end{theorem}

Next we state our result about the low energy asymptotics of $(H^{\scriptscriptstyle A}_\omega)_\omega$ and $(H^{\scriptscriptstyle D}_\omega)_\omega$
and compare it with the asymptotic behaviour of the Laplacian $H$ on the full Cayley graph.
Here and in the sequel we restrict ourselves to the subcritical phase of (site or bond) percolation,
i.e.~we consider a percolation parameter $p < p_c$.
The asymptotic behaviour of the IDS of the adjacency and the Dirichlet percolation Laplacian on a Cayley graph
at low energies is as follows:

\begin{theorem}\label{thm: main Dirichlet}
Let $G$ be a $k$-regular Cayley graph of an amenable, finitely generated group $\Gamma$. Let $(H^{\scriptscriptstyle A}_\omega)_\omega$ and $(H^{\scriptscriptstyle D}_\omega)_\omega$
be the adjacency, respectively the Dirichlet percolation Laplacian for subcritical site or bond percolation on $G$.

Then there is a positive constant $a_p$ such that for all positive $E$ small enough we have
\begin{equation}\label{eq: main dirichlet arbitrary}
N^{\scriptscriptstyle D}(E) \leq N^{\scriptscriptstyle A}(E) \leq \exp\bigg(-\frac{a_p}{2} V\Big(\frac{1}{8\sqrt 2 k} E^{-1/2}-1\Big)\bigg).
\end{equation}

Assume that $G$ has polynomial growth and $V(n) \sim n^{d}$. Then there are positive constants
$\alpha_{D}^{+}(p)$ and $\alpha_{D}^{-}(p)$ such that for all positive $E$ small enough
\begin{equation}\label{eq: main dirichlet poly}
e^{-\alpha_{D}^{-}(p)E^{-d/2}} \leq N^{\scriptscriptstyle D}(E) \leq N^{\scriptscriptstyle A}(E) \leq e^{-\alpha_{D}^{+}(p) E^{-d/2}}.
\end{equation}
Assume that $G$ has superpolynomial growth. Then
\begin{equation}\label{eq: main dirichlet super}
\lim_{E \searrow 0} \frac{\ln |\ln \, N^{\scriptscriptstyle D}(E)|}{|\ln E|}
= \lim_{E \searrow 0} \frac{\ln |\ln \, N^{\scriptscriptstyle A}(E)|}{|\ln E|} = \infty.
\end{equation}
\end{theorem}
\begin{remark}
The inequality  $N^{\scriptscriptstyle D}(E) \leq N^{\scriptscriptstyle A}(E) $ in \eqref{eq: main dirichlet arbitrary} and \eqref{eq: main dirichlet poly}
is deduced from the convergence of the finite volume eigenvalue counting functions
to the IDS which is explained in Remark~\ref{r-basicproperties}.
This is slightly inconsistent with our approach that we want to deduce the asymptotic behaviour of the IDS
from the trace formula \eqref{e-IDS} alone. Note however that our proof of Theorem
\ref{thm: main Dirichlet} shows that even \emph{without the use of the finite volume approximation}
the estimate 
$N^{\scriptscriptstyle D}(E) \leq \exp\Big(-\frac{a_p}{2} V\big((8k\sqrt{2E})^{-1}  -1\big)\Big)$ holds. 
In the case of polynomial growth of order $d$ we have the two-sided bounds 
\begin{align*}
e^{-\alpha_{D}^{-}(p)E^{-d/2}} & \leq  N^{\scriptscriptstyle D}(E)  \leq e^{-\tilde\alpha_{D}^{+}(p) E^{-d/2}}
\text{ and } 
\\
e^{-\tilde\alpha_{D}^{-}(p)E^{-d/2}} &\leq N^{\scriptscriptstyle A}(E) \leq e^{-\alpha_{D}^{+}(p) E^{-d/2}} 
\end{align*}

with some  positive constants $\tilde\alpha_{D}^{+}, \tilde\alpha_{D}^{-}$.
Thus even without the knowledge that the IDS has finite volume approximations the
correct asymptotic behaviour of the IDS may be deduced. An analogous remark applies to equation
\eqref{eq: main dirichlet super}.
\end{remark}

\begin{remark}
Theorem~\ref{thm: main Dirichlet} is a generalisation 
of the results in \cite{KloppN-03} and \cite{KirschM-06} on subcritical bond percolation on the lattice.
Actually, \cite{KloppN-03} treats random hopping models, 
of which the edge percolation model is just a special case; moreover it covers supercritical bond percolation on the lattice
as well. However, \cite{KloppN-03} gives only the upper bound on the IDS (which, in Euclidean geometries, is considered the harder inequality), 
but does not supply a lower bound. 	In \cite{KirschM-06} upper and lower bounds are given using an independent proof.
Theorem~\ref{thm: main Dirichlet} is consistent with the Lifshitz asymptotics for various other types of random Schr\"odinger operators in Euclidean space,
cf.~e.g.~\cite{KirschM-07}. In particular, Lifshitz tails have been proven for the Anderson model,
i.e.~the discrete random Schr\"odinger operator on $\ell^2(\ZZ^d)$ with an i.i.d.~potential.
The first proofs of this result were given in \cite{Mezincescu-85,Simon-85b}.
Let us note that \cite{KloppN-03} shows that eigenvalue inequalities for
some models with off-diagonal disorder can be reduced to the analogous inequalities for the Anderson model.
Thus, Lifshitz type estimates for Anderson models imply those also for bond (and site) percolation models. 
This suggests that one should derive Lifshitz asymptotics for the Anderson model on general graphs, rather than study percolation models.
However, it is not clear whether the proof of \cite{Mezincescu-85,Simon-85b} can be adapted to  general amenable Cayley
graphs. One obstacle for this extension is the fact one has to bound the IDS in terms
of the eigenvalues of the random operator restricted to
a finite graph. In Euclidean space this can be
established using the fact that cubes are a very neat F\o lner sequence which are at the same time
fundamental domains of sublattices.
For general amenable groups such sequences do not exist necessarily.
The other reason is that the eigenvalue estimates for the Anderson models restricted on finite graphs
are established using the Temple inequality, which in turn to be applied efficiently needs lower bounds on the distance
between the two lowest eigenvalues. This lower bound on the spectral gap is immediate in the Euclidean case, while
for more general transitive graphs it may be inferred from a strengthened version of the
Cheeger inequality. Taking these considerations into account
one may hope that the proof of Lifshitz tails for the Anderson model on $\ZZ^d$ can be adapted
for Cayley graphs of polynomial growth. They are, apart from being amenable, residually finite and thus admit
an approximation by finite transitive graphs.
\end{remark}

\begin{remark} 
Actually the statements of the Theorems~\ref{thm: full graph} and~\ref{thm: main Dirichlet} hold 
not only for Laplacians but also for more general symmetric transition operators associated to 
Markov chains on the group $\Gamma$. A precise formulation of these results is presented in 
Section~\ref{s-related}.
\end{remark}

\begin{remark}[Test functions and lower bounds on the IDS]
Let us comment on the fact that in \eqref{eq: main dirichlet arbitrary}
a lower bound of the same type as the upper bound is missing. 
For random operators in Euclidean space the upper bound on the IDS is considered the 
non-trivial part of the Lifshitz asymptotics, while the lower bound can be obtained by a
natural choice of test functions for the Rayleigh quotient. For operators on Cayley graphs the situation is similar,
if we restrict ourselves to polynomial volume growth. Namely, in that case we can match the 
upper bound with a lower bound of the same type, by the use of an appropriate test function.
For groups of super-polynomial volume growth it is not clear whether this can be achieved in general.
In that situation the choice of a test function becomes intricate since 
the leading contribution to the Rayleigh quotient may come from the boundary.
It turns out that Lamplighter groups $\mathbb{Z}_{m} \wr \mathbb{Z}$ have certain special properties which 
enable one to find effective test functions and thus 
establish a proper lower bound on the IDS. These groups are amenable, but of exponential growth.
\end{remark}

Theorem~\ref{thm: main Dirichlet} implies in particular that the IDS is very sparse near the bottom of the spectrum
$E=0$ and consequently zero is a fluctuation boundary.
Relation \eqref{eq: main dirichlet poly} implies that in the case of polynomial growth the
Lifshitz exponent coincides with the van Hove exponent of the Laplacian on the full Cayley graph.
In particular, we have
\begin{equation*}
\lim_{E \searrow 0} \frac{\ln|\ln \, N^{\scriptscriptstyle D}(E)| }{|\ln N_\per(E)|} = \lim_{E \searrow 0} \frac{\ln|\ln \, N^{\scriptscriptstyle A}(E)| }{|\ln N_\per(E)|}\   =1
\end{equation*}

In the case of superpolynomial growth we have that both exponents are infinite.
One may ask whether the sequencs defining them  diverge at the same rate 
and whether the relation
\begin{equation*}
\lim_{E \searrow 0} \frac{\ln\ln|\ln \, N^{\scriptscriptstyle D}(E)| }{\ln|\ln N_\per(E)|} =
\lim_{E \searrow 0} \frac{\ln\ln|\ln \, N^{\scriptscriptstyle A}(E)| }{\ln|\ln N_\per(E)|}\   =1
\end{equation*}
holds under appropriate conditions, for instance, assuming exponential volume growth. 
We are not able prove this for arbitrary groups of exponential  growth, 
but at least for the case of the Lamplighter groups $\mathbb{Z}_{m} \wr \mathbb{Z}$.

\bigskip

\begin{theorem}\label{thm: lamplighter deterministic}
Let $G$ be a Cayley graph of the Lamplighter group $\mathbb{Z}_{m} \wr \mathbb{Z}$. There are
positive constants $a_{1}^{+}$ and $a_{2}^{+}$ such that
$$
N_\per(E) \leq a_{1}^{+}e^{-a_{2}^{+}E^{-1/2}}, \textrm{ for all }
E \textrm{ small enough.}
$$
Moreover for every $r>1/2$ there are positive constants
$a_{r,1}^{-}$ and $a_{r,2}^{-}$ such that
$$
N_\per(E) \geq a_{r,1}^{-}e^{-a_{r,2}^{-}E^{-r}}, \textrm{ for all
} E \textrm{ small enough.}
$$
\end{theorem}

Thus we have an exponential behaviour of the IDS at the bottom of the spectrum,
in particular:
\begin{equation}
\label{eq: Lamplighter exponent}
\lim_{E\searrow 0} \frac{\ln |\ln N_\per(E)|}{|\ln E|} = \frac{1}{2}.
\end{equation}

Now we turn to random operators on Lamplighter groups.

\begin{theorem}\label{thm: lamplighter upper bounds}
Let $G$ be an arbitrary Cayley graph of the Lamplighter group
$\mathbb{Z}_{m} \wr \mathbb{Z}$. For every $p<p_{c}$ there are
positive constants $b_{1}, b_{2}, c_{1}, c_{2},$ such that the IDS of
the adjacency and Dirichlet (site or bond) percolation Laplacian satisfies the following inequality
\begin{equation}\label{eq: lamplighter bounds}
e^{-c_{1}e^{c_{2}E^{-1/2}}} \leq
N^{\scriptscriptstyle D}(E) \leq N^{\scriptscriptstyle A}(E) \leq e^{-b_{1}e^{b_{2}E^{-1/2}}}, \textrm{
for all } E>0 \textrm{ small enough.}
\end{equation}
\end{theorem}

\begin{remark}
 \label{r-lowerbounds}
Our proofs show that the lower bounds
$e^{-\alpha_{D}^{-}(p)E^{-d/2}} \leq N^{\scriptscriptstyle D}(E) \leq N^{\scriptscriptstyle A}(E) $
in Theorem~\ref{thm: main Dirichlet}
and
$e^{-c_{1}e^{c_{2}E^{-1/2}}} \leq  N^{\scriptscriptstyle D}(E) \leq N^{\scriptscriptstyle A}(E)$
in Theorem~\ref{thm: lamplighter upper bounds}
are valid for all values of the percolation parameter $p\in ]0,1]$
\end{remark}

Let us now turn to combinatorial Laplacians $(H^{\scriptscriptstyle N}_\omega)_\omega$, i.e.~Laplacians with the third type of boundary term which we did not discuss yet.
In the case of Neumann boundary conditions the energy zero is not a fluctuation boundary.
The IDS  has a discontinuity at zero, thus one may say that the zeroth $L^2$-Betti number of the random operator $(H^{\scriptscriptstyle N}_\omega)_\omega$
does not vanish. For the combinatorial Laplacian we are able to treat general quasi-transitive graphs.
In particular, $G$
does not need to be neither amenable nor a Cayley graph.
The following generalises a result of \cite{KirschM-06} on $\ZZ^d$-bond percolation.

\begin{theorem}\label{thm: main Neumann}
Let $G$ be a infinite graph with bounded vertex degree and $\Gamma$ a group of automorphisms acting freely and 
cofinitely on $G$. Consider the IDS of the
Neumann percolation Hamiltonian $(H^{\scriptscriptstyle N}_\omega)_\omega$ of a subcritical site or bond percolation process.
There exist positive constants $\alpha_{N}^{+}(p)$ and $\alpha_{N}^{-}(p)$ such that for all positive $E$
small enough
\begin{equation}\label{eq: main neumann}
e^{-\alpha_{N}^{-}(p)E^{-1/2}} \leq N^{\scriptscriptstyle N}(E)-N^{\scriptscriptstyle N}(0) \leq e^{-\alpha_{N}^{+}(p) E^{-1/2}}.
\end{equation}
\end{theorem}

The value $N^{\scriptscriptstyle N}(0)$ coincides with the average number of clusters per vertex in the random graph $G_\omega$.
After subtracting this value we can speak of \eqref{eq: main neumann} as a kind of `renormalised'
Lifshitz asymptotics with exponent $1/2$.

\begin{remark} 
Again, Theorem~\ref{thm: main Neumann}  can be extended to more general models.
More precisely, one can replace the Laplacian $H^{\scriptscriptstyle N}_\omega$ by a regularised Markov transition operator 
in which case the estimates \eqref{eq: main neumann} still hold. Furthermore, 
it is possible to establish the same result for random combinatorial Laplacians generated 
by a long range bond percolation process on a quasi-transitive graph $G$.
Both generalisations are presented in Section~\ref{s-related}.
\end{remark}

On an abstract level Theorem~\ref{thm: main Neumann} and its proof show that the low energy
asymptotics of the combinatorial Laplacian does not depend on geometric properties of $G$,
but only on the rate at which the linear clusters are produced by the percolation process, 
see  \S\ref{ss-abstract-result} for a discussion of this phenomenon.
In this context let us note that M\"uller and Richard \cite{MR} have obtained
results on the low energy asymptotics of combinatorial percolation Laplacians on 
certain Delone sets in $\RR^d$.

\section{Abstract upper and lower bounds on the IDS}
\label{s-abstract-bounds}

To obtain upper bounds for the integrated density of states near
the lower spectral edge, we have to prove that the spectrum is
relatively scarce in this area. In the subcritical phase the
spectrum is only pure point and consists of the eigenvalues of the
operators $H^{\scriptscriptstyle \#}(G')$, where $G'$ goes over the set of all finite
subgraphs. So what one really needs are certain lower bounds for
the eigenvalues of the operators $H^{\scriptscriptstyle \#}(G')$, $\# \in
\left\{N,A,D\right\}$. Vice versa for lower bounds for the IDS we
shall need upper bounds for these eigenvalues in some neighbourhood
of the lower spectral edge. In this spirit we present Propositions
\ref{prp: abstract upper bounds for IDS} and~\ref{prp: abstract
lower bounds for IDS}, which are  generalisations of
Lemmata 2.7 and 2.9 in \cite{KirschM-06}.
Denote with
$\lambda^{\scriptscriptstyle \#}(G')$ the lowest nonzero eigenvalue of the operator
$H^{\scriptscriptstyle \#}(G')$, $\# \in \left\{N,A,D\right\}$.

\begin{proposition}\label{prp: abstract upper bounds for IDS}
Let $G$ be a quasi-transitive graph and ${\scriptstyle\#}\in
\left\{A,D,N\right\}$. Assume that there is a continuous strictly
decreasing function $f \colon \left[1,\infty\right[ \to
\mathbb{R}^{+}$ such that $\lim_{s \to\infty}f(s)=0$ and
$\lambda^{\scriptscriptstyle \#}(G') \geq f(|G'|)$ for any finite subgraph $G'$. Then, for
every $0<p<p_{c}$ there is a positive constant $a_{p}$ such that
\begin{equation} \label{eq: upper bounds for IDS 1}
N^{\scriptscriptstyle \#}(E)-N^{\scriptscriptstyle \#}(0) \leq e^{-a_{p}f^{-1}(E)},
\end{equation}
for all $E$ from the interval $]0,f(1)[$ on which the inverse
function $f^{-1}$ is well defined. 
\end{proposition}

\begin{proof}
Fix $\# \in \left\{N,A,D\right\}$ and $0<E<f(1)$. Since the
subspace $\ell^{2}(C_{x}(\omega))$ is invariant for the operator
$H^{\scriptscriptstyle \#}_{\omega}$ and the restriction on this subspace is exactly
$H^{\scriptscriptstyle \#}(C_{x}(\omega))$ we can write
\begin{equation} \label{eq: formula for IDS}
N^{\scriptscriptstyle \#}(E) - N^{\scriptscriptstyle \#}(0) = \frac{1}{|\mathcal{F}|} \sum_{x \in
\mathcal{F}} \mathbb{E}\Big(\left<\delta_{x},
\chi_{\left]0,E\right]}(H^{\scriptscriptstyle \#}(C_{x}(\omega)))\delta_{x}\right>\Big).
\end{equation}
Now $\chi_{\left]0,E\right]}(H^{\scriptscriptstyle \#}(C_{x}(\omega)))$ is the zero
operator if $E < \lambda^{\scriptscriptstyle \#}(C_{x}(\omega))$, in particular in the case
$|C_{x}(\omega)| < f^{-1}(E)$. Since $\left<\delta_{x},
\chi_{\left]0,E\right]}(H^{\scriptscriptstyle \#}(C_{x}(\omega)))\delta_{x}\right>
\leq 1$ for any $E$ and $\omega$ we can write
\begin{multline*}
N^{\scriptscriptstyle \#}(E)-N^{\scriptscriptstyle \#}(0) 
\\
 = \frac{1}{|\mathcal{F}|} \sum_{x \in
\mathcal{F}} \mathbb{E}\Big( \left<\delta_{x},
\chi_{\left]0,E\right]}(H^{\scriptscriptstyle \#}(C_{x}(\omega)))\delta_{x}\right>
\chi_{\left\{|C_{x}(\omega)| \geq f^{-1}(E)\right\}}(\omega)\Big)
\\  
\leq \frac{1}{|\mathcal{F}|} \sum_{x \in \mathcal{F}}
\mathbb{P}(|C_{x}(\omega)| \geq f^{-1}(E)).
\end{multline*}
Now the result follows from the fact that the probabilities of
large subcritical clusters in quasi-transitive graphs decay
 exponentially, i.e.~$\mathbb{P}(|C_{x}(\omega)| \geq n)
\leq e^{-a_{p}n}$ for all positive integers $n$, all vertices $x$
and all $p<p_{c}$, where $a_{p}$ is a positive constant depending
only on the value of the parameter $p$. This fact is established
for general quasi-transitive graphs in \cite{AntunovicV-a} using the
methods developed in \cite{AizenmanN-84,AizenmanB-87}.
\end{proof}

\begin{proposition}\label{prp: abstract lower bounds for IDS}
Let $G$ be a graph with bounded vertex degree and $\Gamma$ a group of automorphisms
acting cofinitely on $G$ and ${\scriptstyle\#}\in
\left\{A,D,N\right\}$. Suppose that there is a sequence of
connected subgraphs $(G'_{n})_{n}$ and a sequence $(c_{n})_{n}$ in
$\mathbb{R}^{+}$ such that
\begin{enumerate}[\rm(i)]
    \item $\displaystyle \lim_{n \to \infty} |G'_{n}| = \infty$,
    \item $\displaystyle \lim_{n \to \infty}c_{n} = 0$,
    \item $\lambda^{\scriptscriptstyle \#}(G'_{n}) \leq c_{n}$.
\end{enumerate}
For every $E>0$ small enough define $n(E):=\min\left\{n; c_{n}
\leq E \right\}$. Then for every $0<p<1$ there is a positive
constant $b_{p}$ such that the following inequality holds for all
$E>0$ small enough
\begin{equation}\label{eq: lower bounds for IDS 1}
N^{\scriptscriptstyle \#}(E)-N^{\scriptscriptstyle \#}(0) \geq  \frac{1}{|\mathcal{F}|}
\mathbb{P}(G_{n(E)}' \textrm{ is a cluster in } G_\omega) \geq
e^{-b_{p}|G'_{n(E)}|}.
\end{equation}
\end{proposition}

\begin{proof}
Fix $\# \in \left\{N,A,D\right\}$ and $E>0$ small enough so that
$n(E)$ is well defined. Define $\mathcal{S}_{x}(E):=\left\{\tau
\in \Gamma; x \in \tau G_{n(E)}'\right\}$ where $\tau G_{n(E)}'$
is the translation of the subgraph $G_{n(E)}'$ obtained by mapping
each vertex of the subgraph $G_{n(E)}'$ by the automorphism $\tau$. On
the set $\mathcal{S}_{x}(E)$ define the equivalence relation
$\simeq$ in the following way $\tau_{1} \simeq \tau_{2}
:\Leftrightarrow \tau_{1} G_{n(E)}' = \tau_{2} G_{n(E)}'$. Now
take $\mathcal{T}_{x}(E)$, a subset of $\mathcal{S}_{x}(E)$, which
contains exactly one element from each equivalence class. 
Formula \eqref{eq: formula for IDS} implies immediately
\begin{multline} 
|\mathcal{F}|\big( N^{\scriptscriptstyle \#}(E) - N^{\scriptscriptstyle \#}(0)\big) 
\\
\geq \sum_{x \in
\mathcal{F}} \mathbb{E}\Big(\left<\delta_{x},
\chi_{]0,E]}(H^{\scriptscriptstyle \#}(C_{x}(\omega)))\delta_{x}\right>
\chi_{\left\{\exists \, \tau \in \mathcal{T}_{x}(E)\colon C_{x}(\omega)=\tau
G_{n(E)}' \right\}}\Big)
\\
\label{eq:lb}
\geq \sum_{x \in
\mathcal{F}} \mathbb{E}\Big(\left<\delta_{x},
\chi_{]0,\lambda^{\scriptscriptstyle \#}(G'_{n(E)})]}(H^{\scriptscriptstyle \#}(C_{x}(\omega)))\delta_{x}\right>
\chi_{\left\{\exists \, \tau \in \mathcal{T}_{x}(E)\colon C_{x}(\omega)=\tau
G_{n(E)}' \right\}}\Big)	
\end{multline}
since  $ \lambda^{\scriptscriptstyle \#}(G'_{n(E)}) \leq c_{n(E)} \leq E$.	
By definition of $\mathcal{T}_{x}(E)$, the last expression in \eqref{eq:lb} equals
\begin{align*}	
&\sum_{x \in \mathcal{F}} \sum_{\tau
\in \mathcal{T}_{x}(E)} \mathbb{E}\Big(\left<\delta_{x},
\chi_{]0,\lambda^{\scriptscriptstyle \#}(G'_{n(E)})]}(H^{\scriptscriptstyle \#}(\tau
G_{n(E)}'))\delta_{x}\right> \chi_{\left\{C_{x}(\omega)=\tau
G_{n(E)}'\right\}}\Big)\nonumber 
\\& 
= 
\sum_{x \in \mathcal{F}} \sum_{\tau \in
\mathcal{T}_{x}(E)}\left<\delta_{x},
U_{\tau}^{-1}\chi_{]0,\lambda^{\scriptscriptstyle \#}(G'_{n(E)})]}(H^{\scriptscriptstyle \#}(G_{n(E)}'))U_{\tau}
\delta_{x}\right> \mathbb{P}(C_{x}(\omega)=\tau
G_{n(E)}')\nonumber \\ & = \sum_{x \in
\mathcal{F}} \sum_{\tau \in \mathcal{T}_{x}(E)}
\left<\delta_{\tau^{-1}x}, \chi_{]0,\lambda^{\scriptscriptstyle \#}(G'_{n(E)})]}(H^{\scriptscriptstyle \#}
(G_{n(E)}'))\delta_{\tau^{-1}x}\right>
\mathbb{P}(C_{\tau^{-1}x}(\omega)=G_{n(E)}')\nonumber
\end{align*}
Here we used the fact that for any subgraph $G'$ and any element
$\tau$ of the  group $\Gamma$, $H^{\scriptscriptstyle \#}(\tau G') =
U_{\tau}^{-1}H^{\scriptscriptstyle \#}(G')U_{\tau}$, where $U_{\tau}$ is a unitary
operator on $\ell^{2}(G)$ defined by $U_{\tau}f(x):=f(\tau x)$.
The operators $U_{\tau}$ have the property $U_{\tau}\delta_{x} =
\delta_{\tau^{-1}x}$. In the following we introduce a sum over all 
possible values of $\tau^{-1}x$ and see that \eqref{eq:lb} is equal to
\begin{align*}
& \sum_{y \in
G_{n(E)}'} \left<\delta_{y}, \chi_{]0,\lambda^{\scriptscriptstyle \#}(G'_{n(E)})]}(H^{\scriptscriptstyle \#}
(G_{n(E)}'))\delta_{y}\right> \mathbb{P}(C_{y}(\omega)= G_{n(E)}')
\sum_{x \in \mathcal{F}} \sum_{\tau \in \mathcal{T}_{x}(E) \atop
y=\tau^{-1}x} 1 .
\end{align*}
Note that for each vertex $y$ in $G_{n(E)}'$ there exist a vertex $x
\in \mathcal{F}$ and an automorphism $\tau \in \mathcal{T}_{x}(E)$
which maps $y$ to $x$, 
thus $\displaystyle \sum_{x \in \mathcal{F}} \sum_{\tau \in \mathcal{T}_{x}(E) \atop
y=\tau^{-1}x} 1 \geq 1$.
It follows that the last displayed expression can be bounded below by
\begin{align*}
& 
\mathbb{P}(G_{n(E)}' \textrm{ is a cluster in } G_\omega)\sum_{y
\in G_{n(E)}'} \left<\delta_{y},
\chi_{]0,\lambda^{\scriptscriptstyle \#}(G'_{n(E)})]}(H^{\scriptscriptstyle \#} (G_{n(E)}'))\delta_{y}\right>
\\& 
\geq \mathbb{P}(G_{n(E)}' \textrm{ is a cluster in } G_\omega).
\end{align*}
In the last step we used the fact that
$\chi_{]0,\lambda^{\scriptscriptstyle \#}(G'_{n(E)})]}(H^{\scriptscriptstyle \#} (G_{n(E)}')$ is a non-trivial
projection and  its trace is equal to the dimension of its
range which is thus greater or equal than one. Since we are considering
independent percolation on a graph of uniformly bounded vertex
degree we can find a positive constant $b_{p}$ depending only on
$p$, such that 
\[
\frac{1}{|\mathcal{F}|} \mathbb{P}(G' \textrm{ is
a cluster in } G_\omega) \geq e^{-b_{p}|G'|} 
\]
holds for any finite subgraph $G'$.
\end{proof}

\section{Bounds on eigenvalues}
\label{s-eigenvalue-bounds}

As we have seen in the previous section, for good upper and lower
bounds for the IDS we need to estimate $\lambda^{\scriptscriptstyle \#}(G')$. Lower bounds
for eigenvalues (which give upper bounds for IDS) which are
sufficient for our purposes can be given in terms of the growth
rate of the group. Recall that $B(n)$ denotes the ball in a Cayley
graph $G$, of radius $n$ around the  unit element $\iota$
and $V(n)$ stands for the volume (the number of vertices) of
$B(n)$. Also define $\phi(t):= \min \left\{n \geq 0; V(n)>t
\right\}$.

\begin{proposition}\label{prop: faber-krahn}
Let $G=(V,E)$ be a Cayley graph of a finitely generated group
$\Gamma$. For every finite connected subgraph
$G'$
\begin{equation}\label{eq: faber-krahn}
\lambda^{\scriptscriptstyle A}(G') \geq \frac{1}{128 } \, \frac{1}{k^2 \, \phi(2|G'|)^{2}}.
\end{equation}
\end{proposition}

\begin{proof}
If we prove that every non-zero $\varphi$ satisfies $\displaystyle
\frac{\left<\varphi, H^{\scriptscriptstyle A}(G')\varphi\right>}{\|\varphi\|^{2}}
\geq \frac{(128 k^2)^{-1}}{\phi(2|G'|)^{2}}$, the
inequality will follow by the mini-max principle, after taking the
infimum over all non-zero $\varphi$. The above inequality follows from
results in \cite{CoulhonSC-93} and \cite{Woess-00}. Namely in the course
of the proof of Proposition 14.1 in \cite{Woess-00}
one proves that for any $\varphi \in \ell^{2}(G)$ with finite support we
have
\begin{equation}\label{eq: faber-krahn proof 1}
\frac{D_{P}(\varphi)}{\|\varphi\|^{2}}
\geq \frac{1}{2\kappa^{2}\mathfrak{f}(|\supp \varphi|)^{2}},
\end{equation}
where $\kappa$ is a positive constant and $\mathfrak{f} \colon \mathbb{N} \to \mathbb{R}$ is non-decreasing and  
such that $\displaystyle \kappa |\partial_{E} A| \geq \frac{|A|}{\mathfrak{f}(|A|)}$ for
all finite subsets of vertices $A$. Here $\partial_{E} A$ is the edge boundary, i.e.~the set
of edges which have one end-vertex in $A$ and the other outside $A$ and $D_{P}$ is the
Dirichlet sum, which in the special case when $P$
defines the nearest neighbour simple random walk satisfies
$\displaystyle D_{P}(\varphi)=\sum_{x \sim_{G} y}|\varphi(x)-\varphi(y)|^{2}$. Note that
(\ref{e-quadratic-forms}) implies the fact that for any finite
subgraph $G'$ of a Cayley graph $G$ and any $\zeta \in
\ell^{2}(G')$ we have $\displaystyle \left<H^{\scriptscriptstyle A}(G')\zeta,\zeta\right> = \sum_{x
\sim _{G} y}|\widetilde{\zeta}(x)-\widetilde{\zeta}(y)|^{2}$, where
$\widetilde{\zeta}$ is an extension of $\zeta$ in $\ell^{2}(G)$ defined by
setting $\widetilde{\zeta}(x)$ to be equal to $0$ for every $x
\notin G'$. Thus the Dirichlet sum considered in \cite{Woess-00} satisfies
\begin{equation}\label{eq: dirichlet sum}
D_{P}(\widetilde{\zeta}) = \left<H^{\scriptscriptstyle A}(G')\zeta,\zeta\right>,
\end{equation}
in the special case where the transition matrix $P$ corresponds to a
simple nearest neighbour random walk on $G$.
On the other hand
Th\'{e}or\`{e}me 1 in \cite{CoulhonSC-93} shows that for any Cayley graph
of a finitely generated group 
\begin{equation}\label{eq: faber-krahn proof 2}
8 k \, |\partial_{V} A| \geq \frac{|A|}{\phi(2|A|)},
\end{equation}
holds for all finite subsets of vertices $A$. (Here $\partial_{V} A$ is the inner vertex boundary
of $A$, i.e.~the set of vertices in $A$ which have a neighbour outside $A$.) Since
$|\partial_{E} A| \geq |\partial_{V} A|$, the conditions of Proposition 14.1 in \cite{Woess-00}
are satisfied with $\mathfrak{f}(n)=\phi(2n)$ and so
(\ref{eq: faber-krahn proof 1}) and (\ref{eq: dirichlet sum})
together imply the desired inequality.
\end{proof}

The role of the subgraphs $G_{n}'$ from Proposition~\ref{prp: abstract
lower bounds for IDS} will be played by the balls $B(n)$. As for the
sequence $c_{n}$ from the same proposition, the next proposition will
give us a candidate.

\begin{proposition}\label{prop: upper eigen bounds}
Let $G=(V,E)$ be a Cayley graph of a finitely generated group with
polynomial growth. Then there exists a positive constant
$\beta_{D}^{+}$ such that for every positive integer $n$ we have
\begin{equation}\label{eq: upper eigen bounds}
\lambda^{\scriptscriptstyle D}(B(n)) \leq \frac{\beta_{D}^{+} k}{n^{2}}.
\end{equation}
\end{proposition}

\begin{proof}
From the mini-max principle we know
\begin{equation}\label{eq: upper eigen bounds 1}
\lambda^{\scriptscriptstyle D}(B(n)) \leq \frac{\left<\varphi,
H^{\scriptscriptstyle D}(B(n))\varphi\right>}{\|\varphi\|^{2}},
\end{equation}
for every $\varphi \in \ell^{2}(B(n))$. For a test function
$\varphi$ use the radially  symmetric function defined in
the following way:
$$
\varphi(x):=\left\{\begin{array}{l l} n-d(\iota,x), & \textrm{ if }
d(\iota,x) \in \left\{\lfloor n/2 \rfloor, \dots, n\right\} \\
\lceil n/2 \rceil, & \textrm{ if } d(\iota,x)< \lfloor n/2 \rfloor \\
0, & \textrm{ else. }\end{array}\right.
$$
Now we have
\[
\left<\varphi, H^{\scriptscriptstyle D}(B(n))\varphi\right> = \sum_{[x,y] \in E \atop
x,y \in B(n)} |\varphi(x)-\varphi(y)|^{2} + 2 \!\!\!\!\!\!\sum_{(x,y) \in E
\atop x \in B(n), y \notin B(n)}|\varphi(x)-\varphi(y)|^{2} \leq
kV(n)
\]
$$
\|\varphi\|^{2} = \sum_{x \in B(n)}|\varphi(x)|^{2} \geq \lceil
n/2 \rceil ^{2} V(\lfloor n/2 \rfloor).
$$
Inserting these two inequalities into (\ref{eq: upper eigen bounds
1}) and using the fact that $V(n)$ grows polynomially one easily
obtains (\ref{eq: upper eigen bounds}).
\end{proof}

Now we give bounds for the eigenvalues for the combinatorial
Laplacian on quasi-transitive graphs. The first one is a variant of 
the Cheeger inequality 

\begin{proposition}\label{prop: cheeger}
Let $G=(V,E)$ be a quasi-transitive graph with vertex
degree bounded  by $\widetilde k$. For a finite subgraph $G'=(V',E')$ 
denote by $\diam G':=\max_{x,y\in V'} d(x,y)$ its diameter. Then we have
\begin{equation}\label{eq: cheeger}
\lambda^{\scriptscriptstyle N}(G') \geq \frac{2}{  |G'| \, \diam G'}\geq \frac{2}{  |G'|^{2}}.
\end{equation}
\end{proposition}
For the proof see Lemma 1.9 in \cite{Chung-97} or Lemma A.1 in \cite{KhorunzhiyKM-06}.\\

The role of the subgraphs $G_{n}'$ from Proposition~\ref{prp:
abstract lower bounds for IDS}, in the case of the Neumann
Laplacian, will be played by linear subgraphs. A linear subgraph
$L_{n} \subset G$ of length $n$ is the subgraph induced by a path
$v_{1},v_{2},\dots,v_{n+1}$ in the graph $G$, such that the
distance between $v_{i}$ and $v_{j}$ is equal to $|j-i|$, for
every $1 \leq i,j \leq n+1$. Notice that for every connected
infinite graph $G$ and every $n \in \NN$ there exists a linear
subgraph of length $n$ in $G$. To see this fix an arbitrary vertex
$w_{0}$ and take any vertex $w_{n}$ on the sphere of radius $n$
with centre in $w_{0}$ (this sphere is obviously non-empty). Now
take a shortest path $(w_{0}, w_{1}, \dots w_{n-1}, w_{n})$
between the vertices $w_{0}$ and $w_{n}$. Clearly, the vertices
$\left\{w_{0}, w_{1}, \dots w_{n}\right\}$ are vertices of a
linear subgraph $L_{n}$.

\begin{proposition}\label{prop: upper neum eigen bounds}
Let $G=(V,E)$ be a quasi-transitive graph with bounded vertex
degree. For any integer $n$ we have
\begin{equation}\label{eq: upper neum eigen bounds}
\lambda^{\scriptscriptstyle N}(L_{n}) \leq \frac{12}{n^{2}}.
\end{equation}
\end{proposition}

\begin{proof}
We will again use the mini-max principle, i.e.~$\displaystyle
\lambda^{\scriptscriptstyle N}(L_{n}) \leq \frac{\left<\varphi,
H^{\scriptscriptstyle N}(L_{n})\varphi\right>}{\|\varphi\|^{2}}$,  for all
$\varphi \in \ell^{2}(L_{n})$, which are orthogonal to the kernel
of the operator $H^{\scriptscriptstyle N}(L_{n})$. Since the kernel is one
dimensional and contains only constant functions, the condition
that $\varphi$ is orthogonal to the kernel is equivalent to
$\displaystyle \sum_{x \in L_{n}} \varphi (x) = 0$. One obtains
(\ref{eq: upper neum eigen bounds}) by inserting the function
which grows linearly along $L_{n}$ having the value $-n/2$ on one
end-vertex and $n/2$ on the other, see Lemma 2.6 in \cite{KirschM-06}.
\end{proof}

\section{Proofs of the theorems for groups of polynomial growth and quasi-transitive
graphs}\label{s-polynomial-growth} 
We insert the eigenvalue bounds
from the previous section into Propositions~\ref{prp: abstract
upper bounds for IDS} and~\ref{prp: abstract lower bounds for IDS}
to obtain the estimates on the IDS stated in Theorems~\ref{thm:
main Dirichlet} and~\ref{thm: main Neumann}.

\begin{proof}[Proof of Theorem~\ref{thm: main Dirichlet}]
First we prove the general upper bound.
Define the function $\displaystyle g(s):=\frac{1}{128 k^{2}\phi(2s)^{2}}$, which is a right-continuous, non-increasing function which converges to $0$ as $s$ approaches to $\infty$. Now define $\displaystyle g^{*}(E):=\min\left\{s; g(s) \leq E\right\}$. 
We note that $\phi(V(n))=n+1$ for $ n \in \NN$ and estimate
\begin{align}\label{eq: proof main Dirichlet 0}
g^{*}(E)&=\min\left\{s; \phi(2s) \geq \frac{1}{8\sqrt{2}k}E^{-1/2}\right\} \nonumber \\ 
& = \frac{1}{2}V\Big(\Big\lceil \frac{1}{8\sqrt{2}k}E^{-1/2} \Big\rceil -1\Big) \nonumber \\ & \geq \frac{1}{2} V\Big(\frac{1}{8\sqrt{2}k}E^{-1/2}-1\Big).
\end{align}
Now take a sequence of continuous decreasing functions $(f_{n})_{n}$ converging pointwise to $g$
such that for every integer $n$ we have: 
\begin{align*}
f_{n}(s) &\leq g(s) \textrm{ for every positive } s \textrm{ and } 
\\
f_{n}(s) &=g(s) \textrm{ for every } s \textrm{ at which } g \textrm{ is not continuous. }
\end{align*}
Clearly $f_{n}^{-1}(E)=g^{*}(E)$ for every $E$ in the image of $g$ and $\lim_{n \to \infty}f^{-1}_{n}(E) = g^{*}(E)$ if $E$ is not in the image of $g$. Having in view Proposition~\ref{prop: faber-krahn}, 
every $f_{n}$ satisfies the conditions of Proposition~\ref{prp: abstract upper bounds for IDS}. 
Now \eqref{eq: main dirichlet arbitrary} follows.

For the first inequality in (\ref{eq: main dirichlet poly}) use
Proposition~\ref{prp: abstract lower bounds for IDS} with
$G_{n}':=B(n)$ and $c_{n}:=\frac{\beta_{D}^{+} k}{n^{2}}$, where
$\beta_{D}^{+}$ is the constant from Proposition~\ref{prop: upper
eigen bounds}. When $E$ approaches $0$ from above, $n(E)
E^{1/2}$ is bounded from above by a constant and thus the same is true for
$|G_{n(E)}'|E^{d/2}$. Now using the fact that $N^{\scriptscriptstyle D}(0)=0$ the
result follows directly from Proposition~\ref{prp: abstract lower
bounds for IDS}.

For the second inequality in (\ref{eq: main dirichlet poly}) we refer
to Remark~\ref{r-basicproperties}.
By the polynomial growth of $V(n)$ the third inequality follows directly from \eqref{eq: main dirichlet arbitrary}.
%To prove the third inequality we use Proposition~\ref{prp: abstract
%upper bounds for IDS} together with the lower bounds from Proposition
%\ref{prop: faber-krahn}. Because of the polynomial volume growth
%we have $\phi(s) \leq \kappa_{1} s^{1/d}$, for all $s>1$, which
%implies $\displaystyle g(s) :=
%\frac{(128 k^2)^{-1}}{\phi(2s)^{2}} \geq \kappa_{2}
%s^{-2/d}$, for all $s>1$, where $\kappa_{1}, \kappa_{2}$ are
%positive constants. Now the choice $f(s):=\kappa_{2} s^{-2/d}$
%satisfies the conditions of Proposition~\ref{prp: abstract upper
%bounds for IDS} and, since $N^{\scriptscriptstyle A}(0)=0$, the proof is
%straightforward.

\vspace{5 mm}

Now we prove (\ref{eq: main dirichlet super}). By $N^{\scriptscriptstyle D} \leq N^{\scriptscriptstyle A}$
the divergence in (\ref{eq: main dirichlet super}) has to be proven
only for the case of the adjacency Hamiltonian. 
Using \eqref{eq: main dirichlet arbitrary} we get
\begin{equation}\label{eq: main Dirichlet proof 0-1}
\frac{\ln |\ln N^{\scriptscriptstyle A}(E)|}{|\ln E|} \geq \frac{\ln \frac{a_{p}}{2}}{|\ln E|} + \frac{\ln V\Big(\frac{1}{8 \sqrt 2 k} E^{-1/2}-1\Big)}{|\ln E|}.
\end{equation}
In the case of superpolynomial growth we have 
$\displaystyle \lim_{n\to\infty}\frac{\ln V(cn)}{\ln n}=\infty$ for any $c >0$ and thus 
\[
\lim_{E \searrow 0}\displaystyle \frac{\ln V\Big(\frac{1}{8 \sqrt 2 k} E^{-1/2}-1\Big)}{|\ln E|} 
=\infty.
\]
Now the claim follows from \eqref{eq: main Dirichlet proof 0-1}.

\end{proof}

\begin{proof}[Proof of Theorem~\ref{thm: main Neumann}]
By Proposition~\ref{prop: cheeger} we see that the assumptions of
Proposition~\ref{prp: abstract upper bounds for IDS} are satisfied
with $\displaystyle f(s):=\frac{2}{s^{2}}$. Moreover,
by Proposition~\ref{prop: upper neum eigen bounds} we can use
Proposition~\ref{prp: abstract lower bounds for IDS} with
$G_{n}'=L_{n}$ and $\displaystyle
c_{n}=\frac{12}{n^{2}}$. Now the bounds in (\ref{eq: main
neumann}) follow directly.
\end{proof}

As for the periodic case, the formulae for the limits in Theorem~\ref{thm:
full graph} are not new, see for instance Lemma 2.46.~in
\cite{Lueck-02}.  There the operators under consideration are introduced
using the language of homological algebra.  The idea of the proof is to use
a Tauber-type Lemma to turn the return probability estimates from \cite{Varopoulos-87}
into bounds for the IDS.
Let us be a bit more detailed:
Consider the scaled adjacency operator $\frac{1}{k}A(G)$ and its
integrated density of states $N_{\frac{1}{k}A}$. Denote by $\mathbb{P}_{o}(X_{n}=o)$ the return
probability after $n$ steps of the simple nearest neighbour random walk $(X_{n})$
which started at $o$. It follows that
$\mathbb{P}_{o}(X_{n}=o)
=\int_{\mathbb{R}}t^{n}dN_{\frac{1}{k}A}(t)$. Now it is possible
to give sharp bounds on the behaviour of $N_{\frac{1}{k}A}$ near
the upper spectral edge (i.e.~$E=1$) using estimates of  the return
probabilities of the simple random walk. 
These arguments have the flavour of a Tauberian theorem.
In the case of Cayley graphs of groups with polynomial growth the
probabilities $\mathbb{P}_{o}(X_{n}=o)$ behave like $n^{-d/2}$
(see \cite{Varopoulos-87} or Corollary 14.5 and Theorems 14.12 and 14.19 in
\cite{Woess-00}). Now the desired bounds for $N_{\per}$ follow.

The idea to relate the IDS with the return probabilities of the
simple random walk will be important for studying the same problem
in the case of Lamplighter groups. Here we shall refer to results
in \cite{Oguni-07}.

\section{Estimates for Lamplighter groups}
\label{s-lamplighter}
 In this section we derive upper and lower bounds on the
IDS for a particular class of amenable groups of superpolynomial
growth, namely for Lamplighter groups.

Fix a positive integer $m \geq 2$. The Lamplighter group is
defined as the wreath product $\mathbb{Z}_{m} \wr \mathbb{Z}$.
In other words, elements of the group are ordered pairs
$(\varphi,x)$, where $\varphi$ is a function $\varphi \colon
\mathbb{Z} \to \mathbb{Z}_{m}$ with finite support and $x\in
\mathbb{Z}$. The multiplication is given by $(\varphi_{1},x_{1})
\ast (\varphi_{2},x_{2}) := (\varphi_{1} + \varphi_{2}(\cdot -
x_{1}),x_{1}+x_{2})$.
We shall use the following notation. For $x \in \mathbb{Z}$ let
$\delta_{x}$ denote the function which has value $1$ at $x$ and $0$
everywhere else. The zero function will be denoted by
$\mathbf{0}$.

Lamplighter groups are examples of amenable groups with
exponential growth. It suffices to prove these two properties for
some Cayley graph of the Lamplighter group. Consider the Cayley
graph of the Lamplighter group $\mathbb{Z}_{m} \wr \mathbb{Z}$,
defined with respect to the set of generators $\big\{(\mathbf{0},
\pm 1), (k\delta_{0},0); k \in
\mathbb{Z}_{m} \backslash \left\{0\right\}\big\}$.
To prove amenability one only has to notice that the sequence
of sets
$$
\Big(\left\{(\varphi,x); \supp \varphi \subseteq \left\{-n,\dots,
n\right\}, x \in \left\{-n,\dots,n\right\} \right\}\Big)_{n}
$$
is a F\o lner sequence.
Exponential growth follows directly from the fact that for any
function $\varphi$ with support in $\left\{1,2,\dots,n\right\}$
one is able reach the vertex $(\varphi,n)$, from the zero element
in at most $2n$ steps, and so ball of radius $2n$ has at least
$m^{n}$ elements.

Using the same ingredients as in the case of groups with
polynomial growth we now prove the upper bound in (\ref{eq:
lamplighter bounds}).

\begin{proof}[Proof of the upper bound from Theorem~\ref{thm: lamplighter upper bounds}]
Using the fact that the growth of the Lamplighter group is exponential,
%more precisely $V(2n) \ge m^n$,  
the upper bound from \eqref{eq: lamplighter bounds}
follows directly from \eqref{eq: main dirichlet arbitrary}.

%Using the fact that the growth of the Lamplighter group is
%exponential, we get the upper bound $\phi(s) \leq \mu_{1} \ln s$,
%for all $s>2$ and some $\mu_{1}>0$, where the function $\phi$ was
%defined before Proposition~\ref{prop: upper eigen bounds}. It
%implies the lower bound $\displaystyle
%g(s)=\frac{(128 k^2)^{-1}}{\phi(2s)^{2}} \geq \mu_{2}
%(\ln s)^{-2}$, for all $s>2$, where $\mu_{2}$ is a positive
%constant. Thus the function $f(s):=\mu_{2}(\ln s)^{-2}$ satisfies
%the conditions of Proposition~\ref{prp: abstract lower bounds for
%IDS} and gives the desired estimate.
\end{proof}

The lower bound in Theorem~\ref{thm: lamplighter upper bounds}
requires an additional step.  In the proof we shall first
prove the claimed estimate in the case of a particular generator
set and then we shall show how to generalise the result to
arbitrary Cayley graphs.

For an arbitrary generator set $\KKK$ of $\mathbb{Z}_{m} \wr
\mathbb{Z}$ denote by $(\mathbb{Z}_{m} \wr \mathbb{Z})_{\KKK}$ the
Cayley graph induced by the generator set $\KKK$. Also if $V'$ is a
subset of $\mathbb{Z}_{m} \wr \mathbb{Z}$ denote by $G(V',\KKK)$ the
subgraph of $(\mathbb{Z}_{m} \wr \mathbb{Z})_{\KKK}$ induced by the
vertex set $V'$.

Define the following symmetric set of generators
\begin{equation}\label{eq: generators for the horocyclic}
\KKK_{0}:=\left\{(l\cdot\delta_{1},1), l \in \mathbb{Z}_{m}\right\}
\cup \left\{(l\cdot\delta_{0},-1), l \in \mathbb{Z}_{m}\right\}.
\end{equation}
As explained in Section 2 of \cite{Woess-05} the Cayley graph
$(\mathbb{Z}_{m}\wr\mathbb{Z})_{\KKK_{0}}$ is the horocyclic product of two $(m+1)$-regular trees.
We will briefly sketch the necessary definitions and results. For a
comprehensive introduction and a graphical illustration of horocyclic
products of trees we refer to \cite{BartholdiW-05}.

Let $T=(V,E)$ be a $(m+1)$-regular rooted tree with graph metric
$d$. Let $\xi$ be an arbitrary but fixed end. (In the case of trees,
an end is an infinite 
path from the root $o$ in which vertices do not repeat.) For each
vertex $x$ there is the unique path $\gamma_{x}$ from $o$ to $x$.
Denote the intersection of the paths $\gamma_{x}$ and $\xi$, that is the sequence of edges which lie
both in $\gamma_{x}$ and $\xi$, by $\gamma_{x} \cap \xi$. Now the \emph{Busemann function}
of the tree $T$ (with respect to the root $o$ and the end $\xi$)
is defined as $\mathfrak{h} \colon V \to \mathbb{Z}$,
$\mathfrak{h}(x):=|\gamma_{x}|-2|\gamma_{x}\cap\xi|$. For two
vertices $x$ and $y$ which satisfy $\mathfrak{h}(y) \geq
\mathfrak{h}(x)$ and $d(x,y)=\mathfrak{h}(y) - \mathfrak{h}(x)$ we
shall write $x \leq y$.

Assume now that we are given two $(m+1)$-regular trees $T_{1}$ and
$T_{2}$ with Busemann functions $\mathfrak{h}_{1}$ and
$\mathfrak{h}_{2}$ respectively. The \emph{horocyclic product} of
the trees $T_{1}$ and $T_{2}$ is defined as the graph whose vertex
set is given by $\left\{(x_{1},x_{2}); x_{i} \in T_{i},
\mathfrak{h}(x_{1}) + \mathfrak{h}(x_{2}) = 0\right\}$, with two
vertices $(x_{1},x_{2})$ and $(x_{1}',x_{2}')$ adjacent if $x_{i}$
and $x_{i}'$ are adjacent in $T_{i}$ for $i=1,2$. The choice of a
root and an end in the definition is irrelevant since all
horocyclic product of two given trees are mutually isomorphic. As
we mentioned before, the Cayley graph $(\mathbb{Z}_{m} \wr
\mathbb{Z})_{\KKK_{0}}$ is isomorphic to the horocyclic product of
two $(m+1)$-regular trees.

The spectrum of the full Laplace operator on the graph $(\mathbb{Z}_{m} \wr \mathbb{Z})_{\KKK_{0}}$
is pure point, with eigenfunctions having only finite support.
This was shown for the Lamplighter group $\mathbb{Z}_{2} \wr \mathbb{Z}$
in \cite{GrigorchukZ-01} and for more general wreath products in \cite{DicksS-02}.
Here we shall follow the methods from \cite{BartholdiW-05} where the
same facts are proven for Diestel-Leader graphs, which include certain Cayley
graphs of the Lamplighter groups $\mathbb{Z}_{m} \wr \mathbb{Z}$
as a particular case. Moreover, there  the spectrum of the
Laplace operator restricted to certain subgraphs called
tetrahedrons is calculated. This is where the representation of
$(\mathbb{Z}_{m}\wr\mathbb{Z})_{\KKK_{0}}$ as a horocyclic
product becomes essential.

Assume we are given a horocyclic product of two $(m+1)$-regular
trees $T_{1}$ and $T_{2}$ with Busemann functions
$\mathfrak{h}_{1}$ and $\mathfrak{h}_{2}$ and graph metrics
$d_{1}$ and $d_{2}$ respectively. Fix a positive integer $n$ and
take two vertices $x_{1} \in T_{1}$ and $x_{2} \in T_{2}$ such
that $\mathfrak{h}_{2}(x_{2})=-\mathfrak{h}_{1}(x_{1})-n$. Now the
\emph{tetrahedron} $\SSS_n$ with height $n$ is defined as the
subgraph of the horocyclic product of $T_{1}$ and $T_{2}$ induced
by the set of vertices $\left\{(x'_{1},x'_{2})\in T_1\times T_2;
\mathfrak{h}_{1}(x_{1}') + \mathfrak{h}_{2}(x_{2}') = 0, x_{i} \leq x'_{i}, 1=1,2\right\}$. 
Note that we do not
need to specify the vertices $x_{1}$ and $x_{2}$ in the definition
of the tetrahedron, since all tetrahedra with height $n$ are
isomorphic.

Corollary 1 and Proposition 1 from \cite{BartholdiW-05} specify
certain eigenvalues for the Laplacian restricted to tetrahedron
with height $n$ among which is $2m(1-\cos \frac{\pi}{n})$.
Moreover there exist an eigenfunction corresponding to this
eigenvalue which vanishes on the inner vertex boundary of the
tetrahedron, so $2m(1-\cos\frac{\pi}{n})$ is an eigenvalue of the
operators $H^{\scriptscriptstyle \#}(\SSS_n)$ for $\# = N,A,D$. This gives us upper
bounds on the lowest eigenvalue of $H^{\scriptscriptstyle D}(\SSS_n)$ which are precise
enough to lead to the lower bounds for the IDS given in  (\ref{eq:
lamplighter bounds}).

\begin{proof}[Proof of the lower bound from Theorem~\ref{thm: lamplighter upper
bounds}] First we shall consider the Cayley graph $(\mathbb{Z}_{m}
\wr \mathbb{Z})_{\KKK_{0}}$. Again we shall use Proposition~\ref{prp:
abstract lower bounds for IDS}  for $\# = D$. We  set
$G_{n}'=\SSS_n$. It is easy to see that $|\SSS_n|=(n+1)m^{n}$.
Moreover $\displaystyle \lambda^{\scriptscriptstyle D}(\SSS_n) \leq 2m(1-\cos\frac{\pi}{n})
\leq \frac{m\pi^{2}}{n^{2}}$ and thus we can set $\displaystyle
c_{n}=\frac{m\pi^{2}}{n^{2}}$. Proposition~\ref{prp: abstract
lower bounds for IDS} now gives the desired result.

Now take an arbitrary generator set $\KKK$ and consider the
corresponding Cayley graph $(\mathbb{Z}_{m} \wr \mathbb{Z})_{\KKK}$.
Let $V_{n}$ be a set of vertices which induces a tetrahedron with
height $n$ in the Cayley graph $(\mathbb{Z}_{m} \wr
\mathbb{Z})_{\KKK_{0}}$. The same set of vertices need not be
connected in $(\mathbb{Z}_{m} \wr \mathbb{Z})_{\KKK}$ and thus the
induced subgraph in $(\mathbb{Z}_{m} \wr \mathbb{Z})_{\KKK}$ will not
be a good candidate for $G_{n}'$ in Proposition~\ref{prp: abstract
lower bounds for IDS}. For this reason we consider a thickening of
this set defined by $\displaystyle V_{n,R}:=\cup_{x \in
V_{n}}B_{\KKK}(x,R)$, where $B_{\KKK}(x,R)$ is the ball in
$(\mathbb{Z}_{m} \wr \mathbb{Z})_{\KKK}$ of radius $R$ with centre in
$x$. Here $R$ is a positive integer, large enough so that the set
$V_{n,R}$ is connected in $(\mathbb{Z}_{m} \wr \mathbb{Z})_{\KKK}$.
(We can take $R$ equal to the maximal distance in $(\mathbb{Z}_{m}
\wr \mathbb{Z})_{\KKK}$ between vertices which were neighbours in
$(\mathbb{Z}_{m} \wr \mathbb{Z})_{\KKK_{0}}$.) The set $V_{n,R}$
induces a connected subgraph $G(V_{n,R},\KKK)$ of $(\mathbb{Z}_{m}
\wr \mathbb{Z})_{\KKK}$. The volume of $G(V_{n,R},\KKK)$ is bounded
above by a constant times $|V_{n}|=(n+1)m^{n}$, where for the
constant we can take the volume of $B_{\KKK}(x,R)$.

Next we will prove that
\begin{equation}\label{eq: proof generators for horocyclic 1}
\lambda^{\scriptscriptstyle D}(G(V_{n,R},\KKK)) \leq \varrho \lambda^{\scriptscriptstyle A}(\SSS_n),
\end{equation}
for all $n$ and some positive constant $\varrho$. Having in mind
that $2m(1-\cos\frac{\pi}{n})$ is in the spectrum of
$H^{\scriptscriptstyle A}(\SSS_n)$ the desired estimate will follow with the choice
$G_{n}'=G(V_{n,R},\KKK)$ and $c_{n}=\varrho \frac{m\pi^2}{n^{2}}$.

For each function $\varphi \in \ell^{2}(V_{n})$ define the
extension $\widetilde{\varphi}$ to $V_{n,R}$ by setting
$\widetilde{\varphi}(x) = 0$ for all $x \in V_{n,R} \backslash
V_{n}$. Theorem 3.2 in \cite{Woess-00} implies
\begin{equation}\label{eq: proof generators for horocyclic 2}
\left<H^{\scriptscriptstyle D}(G(V_{n,R},\KKK))\widetilde{\varphi},\widetilde{\varphi}\right>
=\left<H^{\scriptscriptstyle A}(G(V_{n,R},\KKK))\widetilde{\varphi},\widetilde{\varphi}\right>
\leq \varrho
\left<H^{\scriptscriptstyle A}(G(V_{n,R},\KKK_{0}))\widetilde{\varphi},\widetilde{\varphi}\right>,
\end{equation}
for some positive constant $\varrho$. 
(To see this consider the special case of Theorem 3.2
in \cite{Woess-00} where the supporting graph is
$(\mathbb{Z}_{m} \wr \mathbb{Z})_{\KKK_{0}}$ and the transition matrix
$P$ defines the nearest neighbour simple random walk on
$(\mathbb{Z}_{m} \wr \mathbb{Z})_{\KKK}$ and use (\ref{eq: dirichlet sum})).

From (\ref{e-quadratic-forms}) and the fact that $\SSS_n =
G(V_{n},\KKK_{0})$ it follows that
\begin{equation}\label{eq: proof generators for horocyclic 3}
\left<H^{\scriptscriptstyle A}(G(V_{n,R},\KKK_{0}))\widetilde{\varphi},\widetilde{\varphi}\right>
= \left<H^{\scriptscriptstyle A}(G(V_{n},\KKK_{0}))\varphi,\varphi\right> =
\left<H^{\scriptscriptstyle A}(\SSS_n)\varphi,\varphi\right>.
\end{equation}
Now, having in mind $\|\varphi\| = \|\widetilde{\varphi}\|$ ,
 (\ref{eq: proof generators for horocyclic 1}) follows from (\ref{eq: proof
generators for horocyclic 2}) and (\ref{eq: proof generators for
horocyclic 3}) and the proof is finished.
\end{proof}

Now we are left to consider the case of the full Laplacian on the
Lamplighter group, i.e.~to prove Theorem~\ref{thm: lamplighter
deterministic}. As we have said before we shall use the relation
between the integrated density of states and return probabilities of
the simple random walk. To simplify expressions we shall use the
following notation. If $f$ and $g$ are two functions $f,g \colon
\mathbb{R}^{+} \to \mathbb{R}$, we shall write $f \preceq g$ if
there exist an $\varepsilon >0$ and positive constants $A$ and $B$
such that $f(x) \leq Ag(Bx)$ for every $x \in ]0,\varepsilon[$.

\begin{theorem}\label{thm: oguni}
Let $G$ be a Cayley graph of a finitely generated amenable group
and $(X_{n})_{n}$ the simple random walk on $G$, started at $o$. Let
$\mathbb{P}_{o}(X_{n}=o)$ be the return probability of the
simple random walk after $n$ steps.
\begin{itemize}
\item[(i)] Assume that there is a constant $0<b<1$ such that for
every positive integer $n$ we have $\mathbb{P}_{o}(X_{2n}=o)
\preceq e^{-(2n)^{b}}$. Then the integrated density of the full
Laplace operator $N_{\per}$ satisfies
$$
N_{\per}(E) \preceq e^{-E^{-\frac{b}{1-b}}}.
$$
\item[(ii)] Assume that there is a constant $0<b<1$ such that
$e^{-(2n)^{b}} \preceq \mathbb{P}_{o}(X_{2n}=o)$. Then, for every
$r > \frac{b}{1-b}$ we have
$$
e^{-E^{-r}} \preceq N_{\per}(E).
$$
\end{itemize}
\end{theorem}

\begin{proof}
The proof of both parts is a minor modification of the proof of
the Theorem 4.4 (parts (ii) and (iii)) in \cite{Oguni-07}. Using
the notation in \cite{Oguni-07} we shall explain the adjustments
which are needed to obtain Theorem~\ref{thm: oguni} from the proof
of \cite[Thm.~4.4]{Oguni-07}. The results in \cite{Oguni-07} are
formulated in terms of a certain distribution function $F$. First
note that the value $F(\lambda)$, for any given positive
$\lambda$, is nothing but $\displaystyle 1- \lim_{s \nearrow
1-\lambda} N_{\frac{1}{k} A}(s)$, where $N_{\frac{1}{k} A}$ is the
IDS of the rescaled adjacency operator $\frac{1}{k}A$. Here $k$ is
the vertex degree in the graph. From the relation $\displaystyle
N_{\per}(\lambda)=1-\lim_{s \nearrow 1- \frac{1}{k} \lambda}
N_{\frac{1}{k} A}(s)$ it is clear that $N_{per}(\lambda)=F(
\lambda/k)$. Thus it is sufficient to prove the desired inequalities
for the function $F$. 

In the proof of the part (ii) we choose
\[
n_{\lambda}:=\left[\left[\Big(\frac{Cb}{\ln
(\frac{1}{1-\lambda})}\Big)^{1/(1-b)}\right]\right] \, 
\]
which replaces the choice 
\[
n_{\lambda}:=\left[\left[\left(\frac{1}{\lambda}\right)^{1/(1-b+\varepsilon)}
\right]\right]
\]
in \cite{Oguni-07}. This enables us to eliminate
the variable $\varepsilon$ from the calculations and to prove the
wanted upper bound for
$F(\lambda)$.

For the lower bounds notice that our assumptions are somewhat
different than those in the part (iii) of the Theorem 4.4 in
\cite{Oguni-07}. Namely we assume uniform lower bounds for the
return probabilities. Following the steps of the cited proof, one can
prove the same inequalities for all positive $\lambda$ small
enough (i.e.~we do not need to define the sets $\Lambda_{C}$).
This is exactly what we wanted.
\end{proof}

\begin{proof}[Proof of Theorem~\ref{thm: lamplighter
deterministic}] Since the return probabilities of the simple
random on any Cayley graph of the Lamplighter group
$\mathbb{Z}_{m} \wr \mathbb{Z}$ satisfy the conditions from both
parts of the preceding theorem with $b=1/3$ (see Theorem 15.15 in
\cite{Woess-00}), the proof is straightforward from Theorem
\ref{thm: oguni}.
\end{proof}

\section{Related models}
\label{s-related}
In this last section we shall present several generalisations of the theorems 
in Section~\ref{s-results}. The first one concerns the case where the 
adjacency operator on $G$ is replaced by a general symmetric transition operator $P$.
It corresponds to a Markov chains whose state space is the vertex set of a Cayley graph.
The percolation process on $G$ leads then to a collection $H^P_\omega, \omega \in \Omega$
of random operators for which we characterise the low energy asymptotics.
We consider also a regularised version $H^R_\omega, \omega \in \Omega$
of the transition operator restricted to the percolation subgraph.
In the case of the Laplacian this regularisation corresponds to Neumann boundary conditions.

The second generalisation concerns combinatorial Laplacians on random  sub-graphs 
generated by a long range bond percolation process on a quasi-transitive graph.

Finally we discuss the spectral asymptotics of
combinatorial Laplacians on an abstract ensemble of percolation graphs.

\subsection{General symmetric transition operators}
We consider now operators which correspond to general transition operators on $\Gamma$,
respectively to Cayley graphs with weights on the edges.

Let $\Gamma$ be a discrete, finitely generated group and 
%$P\colon \ell^2(\Gamma) \to \ell^2(\Gamma)$ a bounded operator 
$\cP$ a matrix indexed by $\Gamma \times \Gamma$ whose coefficients 
$\cP(x,y)$ are non-negative and 
%$X$ a Markov chain on $\Gamma$ whose transition probabilities 
satisfy:
\begin{enumerate}[(a)]
\item the set $S:=\left\{x \in \Gamma; \cP(\iota,x)\neq 0\right\}$ 
is a finite symmetric set of generators of the group $\Gamma$, 
which does not contain $\iota$,
\item for all pairs of group elements $(x,y)$ we have $\cP(y,x)=\cP(x,y)$,
\item for all pairs of group elements $(x,y)$ we have $\cP(x,y)=\cP(\iota, x^{-1}y)$.
\end{enumerate}
\begin{remark}
Note that by (a) and (c) there exists a constant $M \in \RR$ 
such that $\sum_{y\in \Gamma}\cP(x,y)=M$  for all $x \in \Gamma$.
Since the spectral properties of the matrix $\cP(x,y), x,y\in \Gamma$ can be recovered from those 
of $\frac{1}{M}\cP(x,y), x,y\in \Gamma$ we assume in the sequel 
without loss of generality that $\sum_{y\in \Gamma}\cP(x,y)=1$.
Thus the linear map $P \colon \ell^{2}(\Gamma) \to \ell^{2}(\Gamma)$ defined by
\[
(P(x,y)\varphi)(x) = \sum_{y}\cP(x,y)\varphi(y)
\]
is a  \textit{transition operator} whose matrix of transition probabilities is 
given by the coefficients $\cP(x,y), x,y\in \Gamma$.
The Laplace operator corresponding to $P$ is defined as $H^{\scriptscriptstyle P}:=\Id-P$.

The symmetry of the transition probabilities (b) implies the reversibility of the 
Markov chain associated to $P$. More explicitly, there exists a positive function 
$m \colon \Gamma \to \mathbb{R}$ such that $m(x)\cP(x,y)=m(y)\cP(y,x)$ for all pairs of elements $(x,y)$.
\end{remark}

We construct a graph $G_P$ whose vertex set equals $\Gamma$ and such that two vertices $x$ and $y$ are connected 
if and only if $\cP(x,y) \neq 0$. Notice that this graph is actually a Cayley graph of the group $\Gamma$ 
with respect to the generator set $S$. In the particular case in which all probabilities $\cP(\iota,x)$, 
$x \in S$ are the same, the operator $H^{\scriptscriptstyle P}$ is actually equal to $\frac{1}{|S|}H^{\scriptscriptstyle A}(G_P)$.

If $G'=(V',E')$ is a subgraph of $G_P$ we shall denote by $H^{\scriptscriptstyle P}(G')$ the restriction of the operator $H^{\scriptscriptstyle P}$ 
to $\ell^{2}(V')$. 
In other words $H^{\scriptscriptstyle P}(G')$ is defined on $\ell^{2}(V')$ and satisfies $\left<\delta_{x}, H^{\scriptscriptstyle P}(G')\delta_{y}\right> = \left<\delta_{x}, H^{\scriptscriptstyle P}\delta_{y}\right>$ for every two vertices $x$ and $y$ in $V'$.

Now we can run the nearest neighbour independent bond percolation process on the graph $G_P$. 
Each percolation subgraph $G_{\omega}$ will induce a perturbation of the operator $H^{\scriptscriptstyle P}$. Namely we define the operators $H^{\scriptscriptstyle P}_{\omega}:=H^{\scriptscriptstyle P}(G_{\omega})$. In this way we obtain a family of bounded selfadjoint operators, indexed by the set of all possible percolation configurations $\Omega$.

Now we introduce the analog of a percolation Laplacian with Neumann boundary conditions
which corresponds to a general transition operator $P$. 
For a subgraph $G'=(V',E')$ of $G_P$ we define the \textit{regularised Laplacian} as 
\[
 (H^{\scriptscriptstyle R}(G')\varphi)(x) = \sum_{y \in G'; y \sim_{G'}x} \cP(x,y)(\varphi(x)-\varphi(y)), \text{ for every } x \in \ell^{2}(V').
\]
 Now the \textit{regularised percolation Laplacian} is defined as $H^{\scriptscriptstyle R}_{\omega} := H^{\scriptscriptstyle R}(G_{\omega})$.

The quadratic forms of the two operators $H^{\scriptscriptstyle P}(G')$ and $H^{\scriptscriptstyle R}(G')$  are given by
\begin{equation}\label{eq: markov form}
\left<\varphi,H^{\scriptscriptstyle P}(G') \varphi\right> = \sum_{[x,y] \in E'}\cP(x,y)|\varphi(x)-\varphi(y)|^{2} + \sum_{x \in V', [x,y] \notin E'}\cP(x,y)|\varphi(x)|^{2},
\end{equation}
and
\begin{equation}\label{eq: regularized form}
\left<\varphi,H^{\scriptscriptstyle R}(G') \varphi\right> = \sum_{[x,y] \in E'}\cP(x,y)|\varphi(x)-\varphi(y)|^{2}.
\end{equation}

The integrated density of states is again 
defined as 
$$
N^{\scriptscriptstyle \#}(E):=\mathbb{E}\left\{\la \delta_{x},\chi_{]-\infty,E]}(H^{\#}_{\omega})\delta_{x}\ra \right\},
$$ 
for an arbitrary $x \in \Gamma$ and $\# \in \left\{P,R\right\}$. 
Clearly the IDS does not depend on the choice of the vertex $x$ in the definition.
The IDS  of the deterministic operator on the full graph 
is defined as $N_{per}^{P}(E):=\la \delta_{x},\chi_{]-\infty,E]}(H^{\scriptscriptstyle P})\delta_{x}\ra$.

Now we state the  results of this section. They concern the low energy asymptotics of $N^{\scriptscriptstyle P}_\per$, $N^{\scriptscriptstyle P}$, and $N^{\scriptscriptstyle R}$.
Recall that $B(n)$ denotes the ball of radius $n$ around $\iota$ in the graph $G_P$ 
and  $V(n)$ the number of vertices in $B(n)$. 

As in the special case of the Laplacian the first result can be inferred from 
\cite{Woess-00} or \cite{Lueck-02}.
\begin{theorem}\label{thm: main markov deterministic}
Let $\Gamma$ be a finitely generated amenable group. If $\Gamma$ has polynomial growth of order $d$ then 
$$
\lim_{E\searrow 0} \frac{\ln N_{per}^{P}(E)}{\ln E} = \frac{d}{2}.
$$
If $\Gamma$ has superpolynomial growth then
$$
\lim_{E\searrow 0} \frac{\ln N_{per}^{P}(E)}{\ln E} = \infty.
$$
\end{theorem}

\begin{proof}
As explained in Section~\ref{s-polynomial-growth} one can relate the return probabilities of the simple random 
walk and the moments of the measure induced by the IDS. 
For general transition operators $P$ as above
the relation between the return probabilities of the Markov chain $X$ and the moments of the measure induced by the IDS is the same. Thus, just like before, we only have to prove that the return probabilities $\mathbb{P}(X_{n}=\iota)$ behave like $n^{-d/2}$. This actually follows from the same results in \cite{Woess-00} as in Section~\ref{s-polynomial-growth}.
\end{proof}

\begin{theorem}\label{thm: main markov}
Let $\Gamma$ be an amenable, finitely generated group.

Assume that $\Gamma$ has a polynomial growth, i.e.~there exists a positive integer $d$ such that $V(n) \sim n^{d}$. Then for every $p < p_{c}$ there are positive constants
$\alpha_{P}^{+}(p)$ and $\alpha_{P}^{-}(p)$ such that for all positive $E$ small enough
\begin{equation}\label{eq: main markov poly}
e^{-\alpha_{P}^{-}(p)E^{-d/2}} \leq N^{\scriptscriptstyle P}(E) \leq e^{-\alpha_{P}^{+}(p) E^{-d/2}}.
\end{equation}
Assume that $\Gamma$ has superpolynomial growth. Then
\begin{equation}\label{eq: main markov super}
\lim_{E \searrow 0} \frac{\ln |\ln \, N^{\scriptscriptstyle P}(E)|}{|\ln E|} = \infty.
\end{equation}
\end{theorem}

\begin{theorem}\label{thm: main regularized}
Let $\Gamma$ be a finitely generated group. Then for every $p<p_{c}$ there exist positive constants $\alpha_{R}^{+}(p)$ and $\alpha_{R}^{-}(p)$ such that 
\begin{equation}\label{eq: main regularized}
e^{-\alpha_{R}^{-}(p)E^{-1/2}} \leq N^{\scriptscriptstyle R}(E) - N^{\scriptscriptstyle R}(0) \leq e^{-\alpha_{R}^{+}(p) E^{-1/2}}.
\end{equation}
\end{theorem}

\begin{proof}[Proofs of theorems~\ref{thm: main markov} and~\ref{thm: main regularized}]
For both types of the Laplacian propositions~\ref{prp: abstract upper bounds for IDS} and~\ref{prp: abstract lower bounds for IDS} are trivially extended. Thus the problem of finding upper and lower bounds for the IDS is again reduced to the problem of finding lower and upper bounds for the lowest non-zero eigenvalue on finite subgraphs. On the other hand from (\ref{eq: markov form}) and (\ref{eq: regularized form}) one directly obtains 
$$
\Big(\min_{[x,y] \in G_P} \cP(x,y)\Big) H^{\scriptscriptstyle A}(G') \leq H^{\scriptscriptstyle P}(G') \leq \Big(\max_{[x,y] \in G_P} \cP(x,y)\Big) H^{\scriptscriptstyle A}(G') 
$$
and
$$
\Big(\min_{[x,y] \in G_P} \cP(x,y)\Big) H^{N}(G') \leq H^{\scriptscriptstyle R}(G') \leq \Big(\max_{[x,y] \in G_P} \cP(x,y)\Big) H^{N}(G').
$$ 
Since $\min_{[x,y] \in G_P} \cP(x,y)=\min_{y \in S} \cP(\iota,y)$ this term is strictly positive.
By the invariance under the group $\Gamma$ the term $\max_{[x,y] \in G_P} \cP(x,y)$ is finite.

Now the bounds for eigenvalues from propositions~\ref{prop: faber-krahn}, \ref{prop: upper eigen bounds}, \ref{thm: main Dirichlet} and~\ref{thm: main Neumann} (with additional positive multiplication factors $\displaystyle \min_{[x,y] \in G_P} \cP(x,y)$ and $\displaystyle \max_{[x,y] \in G_P} \cP(x,y)$) transfer directly to this generalised setting. Using these bounds, 
the proofs of Theorems~\ref{thm: main markov} and~\ref{thm: main regularized} are completed in the same way as 
the proofs of Theorems~\ref{thm: main Dirichlet} and~\ref{thm: main Neumann} in Section~\ref{s-polynomial-growth}. 
\end{proof}

%\begin{proof}[Proof of Theorem~\ref{thm: main markov}]
%Upper and lower bounds for the eigenvalues are basically the same as in propositions~\ref{prop: faber-krahn} and~\ref{prop: upper eigen bounds}. The derivation of the upper bound is the same as in proposition~\ref{prop: upper eigen bounds} with the difference that in the calculation one has to to pass to the maximal value of the factors $\cP(x,y)$ appearing in the formula for the quadratic form (\ref{eq: markov form}).
%\end{proof}

%\begin{proof}[Proof of Theorem~\ref{thm: main regularized}]
%For the eigenvalue bounds in this case we will take those from propositions~\ref{prop: cheeger} and~\ref{prop: upper neum eigen bounds}. The upper bound and its proof remain completely the same.
%\end{proof}

\subsection{Laplacians on long range percolation graphs}
The long range percolation model is a generalisation of the nearest neighbour model. 
In this model one allows any pair of vertices to be directly connected, i.e. adjacent, in the percolation graph. 
However, to control the size of the percolation clusters, the probabilities that two vertices are directly connected 
must decay as the distance between them converges to infinity. 
More precisely we take an arbitrary quasi-transitive graph $G$ with finite vertex degrees and a fundamental domain $\mathcal{F}$. 
We construct the graph $\overline{G}$ by connecting each pair of vertices in $G$. 
For the metric on the set of vertices of $\overline{G}$ we will take the graph metric $d$ in $G$. 
In particular, two vertices $x,y$ of $\overline{G}$ may be adjacent (directly connected) although $d(x,y) >1$. 

For each pair of vertices $x$ and $y$ we take a positive real number $J_{[x,y]}$ such that
\begin{itemize}
 \item $J_{[\gamma x, \gamma y]} = J_{[x,y]}$, for all vertices $x$ and $y$ and all graph automorphisms $\gamma$,
 \item $J_{x}:=\displaystyle \sum_{y \in G} J_{[x,y]} < \infty$ for all vertices $x$ (we define $\displaystyle J:=\max_{x \in \mathcal{F}}J_{x}$).
\end{itemize}
Now for each edge $e$ in $\overline{G}$, one declares $e$ to be open with probability $1-e^{-\beta J_{e}}$, for some positive parameter $\beta$, independently of all other edges in $\overline{G}$. 
The percolation subgraph $G_{\omega}=(V_{\omega},E_{\omega})$ is defined as the subgraph spanned by the set of open edges. 
$G_{\omega}$ contains arbitrary long edges almost surely.
Notice that the probability that certain edge is open is increasing in $\beta$. 
Thus, the subcritical phase, in which all clusters are almost surely finite corresponds to small values of the parameter $\beta$ and the supercritical phase in which there exists an infinite cluster corresponds to large values of the parameter $\beta$. Just like in the case of the  nearest neighbour percolation model these two phases are separated by a single value of the parameter $\beta$. This value will be denoted by $\beta_{c}$. The cluster containing an arbitrary vertex $x$ will again be denoted by $C_{x}$. 
In \cite{AntunovicV-a} it is proven that the probabilities $\mathbb{P}(|C_{x}| \geq n)$ 
decay exponentially in the subcritical phase, i.e.~$\beta < \beta_c$,
of the long range model, (see \cite{AizenmanB-87} for the case $G = \ZZ^d$).

The percolation Laplacian is defined as the combinatorial Laplacian on the percolation subgraph. 
More precisely we define the operator $H^{\scriptscriptstyle N,L}$ on $\ell^{2}(V_{\omega})$ for all $\varphi$ with finite support by
\[
(H^{\scriptscriptstyle N,L}_{\omega}\varphi)(x) = \sum_{y \in G_{\omega};y \sim_{G_{\omega}} x} (\varphi(x)-\varphi(y)).
\]
% In the subcritical phase this sum is almost surely finite and thus, 
% in this case, the operator $H^{\scriptscriptstyle N,L}_{\omega}$ is well defined. 
% However, 
Since we have no upper bound on the vertex degrees any more, this operator is not bounded almost surely. 
It is still self-adjoint on its maximal domain 
$\mathcal{D}(H^{\scriptscriptstyle N,L}_{\omega}):=\left\{\varphi \in \ell^{2}(V_{\omega}); H^{\scriptscriptstyle N,L}_{\omega}\varphi \in \ell^{2}(V_{\omega})\right\}$.

The integrated density of states is again defined as 
$$
N^{\scriptscriptstyle N,L}(E) := \mathbb{E}\left\{\Tr[\chi_{\mathcal{F}}\chi_{]-\infty,E]}(H^{\scriptscriptstyle N,L}_{\omega})]\right\}.
$$ 
It exhibits the same asymptotics as the combinatorial Laplacian in the nearest neighbour percolation model.

\begin{theorem}\label{thm: main long range}
Let $G$ be a quasi-transitive graph with finite vertex degrees. For every subcritical parameter $\beta$ there exist positive constants $\alpha_{N,L}^{-}(\beta)$ and $\alpha_{N,L}^{+}(\beta)$ such that for all positive $E$ small enough
$$
e^{-\alpha^{-}_{N,L}(\beta)E^{-1/2}} \leq N^{\scriptscriptstyle N,L}(E)-N^{\scriptscriptstyle N,L}(0) \leq e^{-\alpha^{+}_{N,L}(\beta)E^{-1/2}}.
$$
\end{theorem}

\begin{proof}
Similarly as in sections~\ref{s-abstract-bounds} and~\ref{s-eigenvalue-bounds} we are able to prove the following statements:
\begin{itemize}
 \item[1)] Let $f \colon \mathbb{R}^{+} \to \mathbb{R}^{+}$ be a continuous, strictly decreasing function, such that $\lim_{s \to \infty}f(s)=0$ and $\lambda^{N}(G') \geq f(|G'|)$ holds for every finite subgraph $G'$ of $\overline{G}$. Then, for each $\beta < \beta_{c}$, the inequality $N^{\scriptscriptstyle N,L}(E)-N^{\scriptscriptstyle N,L}(0) \leq e^{-a^{L}_{\beta}f^{-1}(E)}$ holds for some positive constant $a^{L}_{\beta}$ and all positive $E$ small enough.
 \item[2)] Assume that there is a sequence of connected subgraphs $(G'_{n})_{n}$ in $\overline{G}$, with the property $\lim_{n \to \infty}|G'_{n}|=\infty$ and a sequence  $(c_{n})_{n}$ in $\mathbb{R}^{+}$ which converges to $0$ such that $\lambda^{N}(G'_{n}) \leq c_{n}$, for all $n$. Furthermore assume there is a positive integer $k$ such that for any $n \in \mathbb{N}$ and any two directly connected vertices $x$ and $y$ in $G'_{n}$ we have $d(x,y) \leq k$. Again define $n(E):=\min\left\{n;c_{n}\leq E\right\}$. Then for every $\beta \in \mathbb{R}^{+}$ there is a positive constant $b^{L}_{\beta}$ such that for all positive $E$ small enough we have
$$
N^{\scriptscriptstyle N,L}(E)-N^{\scriptscriptstyle N,L}(0) \geq \frac{1}{|\mathcal{F}|}
\mathbb{P}(G_{n(E)}' \textrm{ is a cluster in } G_{\omega}) \geq
e^{-b_{\beta}^{L}|G'_{n(E)}|}.
$$
 \item[3)] For every finite subgraph $G'$ of $G$ we have $\lambda^{N}(G') \geq \frac{1}{|G'|^{2}}$.
 \item[4)] For any linear graph $L_{n}$ in $\overline{G}$ (i.e.~a subgraph of $\overline{G}$ with $n+1$ vertices $v_{1},\dots,v_{n+1}$, 
such that $d(v_{i},v_{j})= |j-i|$ and edge set 
$\big\{[v_{i},v_{i+1}], 1 \leq i \leq n \big\}$) we have $\lambda^{N}(L_{n}) \leq \frac{12}{n^{2}}$.
\end{itemize}
From this results our claim follows like in the proof of Theorem~\ref{thm: main Neumann}. Statements 1) and 4) are proven in the same way as propositions~\ref{prp: abstract upper bounds for IDS} and~\ref{prop: upper neum eigen bounds}. 
For statement 3) see Proposition \ref{prop: cheeger}. 
As for Statement 2), the proof proceeds along the same lines as the proof of Proposition~\ref{prp: abstract lower bounds for IDS}. 
However, to bound from below the probability that $G_{n(E)}'$ is a cluster in $G_{\omega}$, we need the additional condition 
that there are no edges longer than $k$ in $G_n'$ and the following lemma.
\end{proof}

\begin{lemma}\label{lemma: probability long range}
For an arbitrary $k \in \mathbb{N}$ and an arbitrary positive $\beta$ there exists a positive constant $\varsigma_{k}$ such that the following statement is true:\\
For every connected subgraph $G'$ of $\overline{G}$ such that the distance between any two directly connected vertices in $G'$ is less or equal than $k$ we have
\begin{equation}\label{eq: probability long range}
\mathbb{P}(G' \textrm{ is a cluster of } G_{\omega}) \geq e^{-\varsigma_{k}|G'|}
\end{equation}
\end{lemma}
\begin{proof}
Let $G'$ be an arbitrary subgraph which satisfies the assumptions of the lemma. We partition the set of edges 
in $G$ adjacent to $x$ into two disjoint subsets: in $I_{x}$ we put those which are edges of $G'$ and in $O_{x}$ others.
It is clear that the probability $\mathbb{P}(G' \textrm{ is a cluster of } G_{\omega})$ can be estimated 
by the product of the probability that all edges in $I_{x}$, $x \in G'$ are open and the probability that all edges in $O_{x}$, $x \in G'$ are closed. Therefore we can write
\begin{align*}
\mathbb{P}(G' \textrm{ is a cluster of } G) &= \mathbb{P}(\bigcap_{x \in G'}\bigcap_{e \in I_{x}}\left\{e \textrm{ is open}\right\}) \mathbb{P}(\bigcap_{x \in G'}\bigcap_{e \in O_{x}}\left\{e \textrm{ is closed}\right\}) \\ & \geq \prod_{x \in G'}\Big(\prod_{e \in I_{x}}\mathbb{P}(e \textrm{ is open})\prod_{e \in O_{x}}\mathbb{P}(e \textrm{ is closed})\Big) \\ & = \prod_{x \in G'}\Big(e^{-\beta J}\prod_{e \in I_{x}}(e^{\beta J_{e}}-1)\Big) \\ & \geq (e^{-\beta J}c)^{|G'|}.
\end{align*}
Here $c$ is defined as $\displaystyle c:=\min_{x \in \mathcal{F}}\min_{A; A \subset (B(x,k)\backslash \left\{x\right\})} \prod_{y \in A}(e^{\beta J_{[x,y]}}-1)$, where $B(x,k)$ is a ball of radius $k$ around $x$. Obviously $c$ is positive and because of the invariance of the parameters $J_{[x,y]}$ under the automorphisms of $G$ we have $\displaystyle \prod_{e \in I_{x}}(e^{\beta J_{e}}-1) \geq c$, for all $x \in G'$. Since the constant $c$ does not depend on $G'$ the claim follows.
\end{proof}

\subsection{An abstract result}
\label{ss-abstract-result}
We have encountered the phenomenon, that in the case of the combinatorial Laplacian
the low energy asymptotics is independent of the volume growth behaviour of the graph.
This is consistent with the results on {E}rd{\" o}s-{R}\'enyi random graphs obtained in \cite{KhorunzhiyKM-06}. 

In the following we present an abstract result which tries to capture this phenomenon
and to single out properties which the stochastic process which generates the random graphs
needs to satisfy to obtain a low energy asymptotics as in Theorem \ref{thm: main Neumann}.

Let $G=(V, E)$ be a graph with countable vertex set $V$ and vertex degree bounded by $\tilde k$. 
Let an independent (site or bond) percolation process on $G$ be given and denote the percolation 
subgraph of $G$ associated to the configuration $\omega\in\Omega$ by $G_\omega$.
Fix a finite subset $\cF$ of $V$ and assume that there exists a
doubly infinite path $\mathfrak{P}$ in $G$ which contains a vertex $o\in\cF$. 
In other words $\mathfrak{P}\colon \ZZ\to V$ is injective and contains $o$ in its image.
Denote by $H^{\scriptscriptstyle N}_\omega$
the combinatorial Laplacian on $G_\omega$ and define the monotone function
\[
N^{\scriptscriptstyle N}(E):= |\cF|^{-1} \EE \left \{ \Tr [\chi_\cF \ \chi_{]-\infty,E]}(H^{\scriptscriptstyle N}_\omega)] \right \}
\]
which in many situations can be interpreted as the IDS. Assume that 
there is no infinite cluster in the graph $G_\omega$ almost surely
and that the cluster size distribution 
decays exponentially, more precisely
\begin{equation}
\label{e-exponential-decay} 
 \PP\{ |C_x(\omega)|\ge n\} \le e^{-a n}
\end{equation}
for some $a>0$ and all $x \in \cF$.
In the case of site percolation assume furthermore that 
$p_a:=\inf_{x\in V} \PP\{x \text{ is open} \}$  and $p_d:=\inf_{x\in V} \PP\{x \text{ is closed} \}$  
are strictly positive. Similarly in the case of bond percolation assume that 
$\inf_{e\in E} \PP\{e \text{ is open} \}$  and $\inf_{e\in E} \PP\{e \text{ is closed} \}$  
are strictly positive. 

\begin{theorem}
Assume the setting described in this paragraph.
Then there exist constants $\alpha_-, \alpha_+>0$ such that for all $E>0$ sufficiently small
\begin{equation}
e^{-\alpha_{-}E^{-1/2}} \leq N^{\scriptscriptstyle N}(E)-N^{\scriptscriptstyle N}(0) \leq e^{-\alpha_{+} E^{-1/2}}.
\end{equation}
 \end{theorem}
\begin{proof}
For the proof of the upper bound one uses the same inequalities as in the proof of  
Proposition \ref{prp: abstract upper bounds for IDS}, together with the eigenvalue estimate in 
Proposition \ref{prop: cheeger} and the exponential decay assumption \eqref{e-exponential-decay}.

For the lower bound one uses that 
\begin{equation}
\label{e-7.3-lower} 
N^{\scriptscriptstyle N}(E)-N^{\scriptscriptstyle N}(0) 
\geq 
|\cF|^{-1} \sum_{j=0}^n \EE \left  \{ \chi_{\Omega_{n,j}} \la \delta_o, \chi_{]-\infty,E]}
(H^{\scriptscriptstyle N}_\omega) \delta_o \ra \right \}
\end{equation}
where $n$ is chosen such that $\frac{12}{n^2}\le E$ and 
$\Omega_{n,j} \subset \Omega$ denotes the set of configurations where 
the cluster $C_o(\omega)$ is a linear cluster $L_n$ and $o$ is the vertex 
at the $j^{\text{th}}$ position of $L_n$. By the assumption on the existence of the infinite path $\mathfrak{P}$
such configurations exist and by the independence assumption we estimate the probability of $\Omega_{n,j}$ from below
by $p_a^n \cdot p_d^{\tilde k n}$.

Now let $\phi_n$ be a normalised  eigenfunction associated to the eigenvalue $\lambda^{\scriptscriptstyle N}(L_n) \le \frac{12}{n^2}$.
Since
\begin{align*}
\sum_{j=0}^n \EE \left  \{ \chi_{\Omega_{n,j}} \la \delta_o, \chi_{]-\infty,E]}
(H^{\scriptscriptstyle N}_\omega) \delta_o \ra \right \}
&\ge
\sum_{j=0}^n \EE \left  \{ \chi_{\Omega_{n,j}}  |\phi_n(o)|^2 \right \} 
\\&=
\sum_{j=0}^n \EE \left  \{ \chi_{\Omega_{n,j}}  |\phi_n(j)|^2 \right \} 
\end{align*}
we have $N^{\scriptscriptstyle N}(E)-N^{\scriptscriptstyle N}(0) \geq |\cF|^{-1}   p_a^n \cdot p_d^{\tilde k n}$.
This completes the proof.
\end{proof}

\def\cprime{$'$}\def\polhk#1{\setbox0=\hbox{#1}{\ooalign{\hidewidth
  \lower1.5ex\hbox{`}\hidewidth\crcr\unhbox0}}}

\end{document}